\documentclass[12pt]{amsart}
\usepackage{amscd,amssymb,latexsym}
\usepackage{fullpage}
\newcommand{\Mdef}[2]{\newcommand{#1}{\relax \ifmmode #2 \else $#2$\fi}}

\newcommand{\cok}{\mathrm{cok}}

\newcommand{\im}{\mathrm{im}}

\newcommand{\sm }{\wedge}

\newcommand{\tensor}{\otimes}

\newcommand{\sdr}{\rtimes}

\newcommand{\Hom}{\mathrm{Hom}}

\newcommand{\Ext}{\mathrm{Ext}}
\Mdef{\bhom}{\mathbf{\hat{H}om}}

\Mdef{\Mod}{\mathrm{mod}}

\newcommand{\st}{\; | \;}

\newtheorem{thm}{Theorem}[section]
\newtheorem{lemma}[thm]{Lemma}
\newtheorem{prop}[thm]{Proposition}
\newtheorem{cor}[thm]{Corollary}

\theoremstyle{definition}

\newtheorem{defn}[thm]{Definition}

\newtheorem{warning}[thm]{Warning}

\newtheorem{example}[thm]{Example}

\newtheorem{remark}[thm]{Remark}

\newtheorem{conj}[thm]{Conjecture}

\newcommand{\qqed}{\qed \\[1ex]}
\renewenvironment{proof}[1][\hspace*{-.8ex}]{\noindent {\bf Proof #1:\;}}{\qqed}

\Mdef{\PH} {\Phi^H}
\Mdef{\PK} {\Phi^K}
\Mdef{\PL} {\Phi^L}
\Mdef{\PT} {\Phi^{\T}}

\Mdef{\ef}{E{\cF}_+}
\Mdef{\etf}{\widetilde{E}{\cF}}
\Mdef{\eg}{E{G}_+}
\Mdef{\etg}{\tilde{E}{G}}

\newcommand{\piA}{\pi^{\cA}}

\Mdef{\infl}{\mathrm{inf}}
\Mdef{\defl}{\mathrm{def}}
\Mdef{\res}{\mathrm{res}}
\Mdef{\ind}{\mathrm{ind}}
\Mdef{\coind}{\mathrm{coind}}

\Mdef{\univ}{\mathcal{U}}

\Mdef{\Fp}{\mathbb{F}_p}
\Mdef{\Zpinfty}{\Z /p^{\infty}}
\Mdef{\Zpadic}{\Z_p^{\wedge}}

\newcommand{\bi}{\begin{itemize}}
\newcommand{\be}{\begin{enumerate}}
\newcommand{\bc}{\begin{center}}
\newcommand{\bd}{\begin{description}}
\newcommand{\ei}{\end{itemize}}
\newcommand{\ee}{\end{enumerate}}
\newcommand{\ec}{\end{center}}
\newcommand{\ed}{\end{description}}
\newcommand{\dichotomy}[2]{\left\{ \begin{array}{ll}#1\\#2 \end{array}\right.}

\newcommand{\adjunction}[4]{
\diagram
#1:#2 \rrto<0.7ex> &&
#3  \llto<0.7ex> :#4 
\enddiagram}
\newcommand{\lra}{\longrightarrow}
\newcommand{\lla}{\longleftarrow}

\newcommand{\sets}{\mathbf{Sets}}
\newcommand{\crings}{\mathbf{Rings}}

\newcommand{\rings}{\mathbf{Rings}}

\Mdef{\we}{\mathbf{we}}
\Mdef{\fib}{\mathbf{fib}}
\Mdef{\cof}{\mathbf{cof}}
\Mdef{\BI}{\mathcal{BI}}

\newcommand{\cofibre}{\mathrm{cofibre}}

\newcommand{\ilim}{\mathop{ \mathop{\mathrm{lim}} \limits_\leftarrow} \nolimits}
\newcommand{\colim}{\mathop{  \mathop{\mathrm {lim}} \limits_\rightarrow} \nolimits}

\Mdef{\A}{\mathbb{A}}
\Mdef{\B}{\mathbb{B}}
\Mdef{\C}{\mathbb{C}}
\Mdef{\D}{\mathbb{D}}
\Mdef{\E}{\mathbb{E}}
\Mdef{\T}{\mathbb{T}}
\Mdef{\F}{\mathbb{F}}
\Mdef{\G}{\mathbb{G}}
\Mdef{\I}{\mathbb{I}}
\Mdef{\N}{\mathbb{N}}
\Mdef{\Q}{\mathbb{Q}}
\Mdef{\R}{\mathbb{R}}
\Mdef{\bbS}{\mathbb{S}}
\Mdef{\Z}{\mathbb{Z}}

\Mdef{\bA}{\mathbb{A}}
\Mdef{\bB}{\mathbb{B}}
\Mdef{\bC}{\mathbb{C}}
\Mdef{\bD}{\mathbb{D}}
\Mdef{\bE}{\mathbb{E}}
\Mdef{\bF}{\mathbb{F}}
\Mdef{\bG}{\mathbb{G}}
\Mdef{\bH}{\mathbb{H}}
\Mdef{\bI}{\mathbb{I}}
\Mdef{\bJ}{\mathbb{J}}
\Mdef{\bK}{\mathbb{K}}
\Mdef{\bL}{\mathbb{L}}
\Mdef{\bM}{\mathbb{M}}
\Mdef{\bN}{\mathbb{N}}
\Mdef{\bO}{\mathbb{O}}
\Mdef{\bP}{\mathbb{P}}
\Mdef{\bQ}{\mathbb{Q}}
\Mdef{\bR}{\mathbb{R}}
\Mdef{\bS}{\mathbb{S}}
\Mdef{\bT}{\mathbb{T}}
\Mdef{\bU}{\mathbb{U}}
\Mdef{\bV}{\mathbb{V}}
\Mdef{\bW}{\mathbb{W}}
\Mdef{\bX}{\mathbb{X}}
\Mdef{\bY}{\mathbb{Y}}
\Mdef{\bZ}{\mathbb{Z}}

\Mdef{\cA}{\mathcal{A}}
\Mdef{\cB}{\mathcal{B}}
\Mdef{\cC}{\mathcal{C}}
\Mdef{\mcD}{\mathcal{D}} 
\Mdef{\cE}{\mathcal{E}}
\Mdef{\cF}{\mathcal{F}}
\Mdef{\cG}{\mathcal{G}}
\Mdef{\mcH}{\mathcal{H}} 
\Mdef{\cI}{\mathcal{I}}
\Mdef{\cJ}{\mathcal{J}}
\Mdef{\cK}{\mathcal{K}}
\Mdef{\mcL}{\mathcal{L}}

\Mdef{\cM}{\mathcal{M}}
\Mdef{\cN}{\mathcal{N}}
\Mdef{\cO}{\mathcal{O}}
\Mdef{\cP}{\mathcal{P}}
\Mdef{\cQ}{\mathcal{Q}}
\Mdef{\mcR}{\mathcal{R}}
\Mdef{\cS}{\mathcal{S}}
\Mdef{\cT}{\mathcal{T}}
\Mdef{\cU}{\mathcal{U}}
\Mdef{\cV}{\mathcal{V}}
\Mdef{\cW}{\mathcal{W}}
\Mdef{\cX}{\mathcal{X}}
\Mdef{\cY}{\mathcal{Y}}
\Mdef{\cZ}{\mathcal{Z}}

\Mdef{\ca}{\mathcal{a}}
\Mdef{\ct}{\mathcal{t}}

\Mdef{\At}{\tilde{A}}
\Mdef{\Bt}{\tilde{B}}
\Mdef{\Ct}{\tilde{C}}
\Mdef{\Et}{\tilde{E}}
\Mdef{\Ht}{\tilde{H}}
\Mdef{\Kt}{\tilde{K}}
\Mdef{\Lt}{\tilde{L}}
\Mdef{\Mt}{\tilde{M}}
\Mdef{\Nt}{\tilde{N}}
\Mdef{\Pt}{\tilde{P}}

\Mdef{\tA}{\tilde{A}}
\Mdef{\tB}{\tilde{B}}
\Mdef{\tC}{\tilde{C}}
\Mdef{\tE}{\tilde{E}}
\Mdef{\tH}{\tilde{H}}
\Mdef{\tK}{\tilde{K}}
\Mdef{\tL}{\tilde{L}}
\Mdef{\tM}{\tilde{M}}
\Mdef{\tN}{\tilde{N}}
\Mdef{\tP}{\tilde{P}}

\Mdef{\ft}{\tilde{f}}
\Mdef{\xt}{\tilde{x}}
\Mdef{\yt}{\tilde{y}}

\Mdef{\Ab}{\overline{A}}
\Mdef{\Bb}{\overline{B}}
\Mdef{\Cb}{\overline{C}}
\Mdef{\Db}{\overline{D}}
\Mdef{\Eb}{\overline{E}}
\Mdef{\Fb}{\overline{F}}
\Mdef{\Gb}{\overline{G}}
\Mdef{\Hb}{\overline{H}}
\Mdef{\Ib}{\overline{I}}
\Mdef{\Jb}{\overline{J}}
\Mdef{\Kb}{\overline{K}}
\Mdef{\Lb}{\overline{L}}
\Mdef{\Mb}{\overline{M}}
\Mdef{\Nb}{\overline{N}}
\Mdef{\Ob}{\overline{O}}
\Mdef{\Pb}{\overline{P}}
\Mdef{\Qb}{\overline{Q}}
\Mdef{\Rb}{\overline{R}}
\Mdef{\Sb}{\overline{S}}
\Mdef{\Tb}{\overline{T}}
\Mdef{\Ub}{\overline{U}}
\Mdef{\Vb}{\overline{V}}
\Mdef{\Wb}{\overline{W}}
\Mdef{\Xb}{\overline{X}}
\Mdef{\Yb}{\overline{Y}}
\Mdef{\Zb}{\overline{Z}}

\Mdef{\db}{\overline{d}}
\Mdef{\hb}{\overline{h}}
\Mdef{\qb}{\overline{q}}
\Mdef{\rb}{\overline{r}}
\Mdef{\tb}{\overline{t}}
\Mdef{\ub}{\overline{u}}
\Mdef{\vb}{\overline{v}}

\Mdef{\hc}{\hat{c}}
\Mdef{\he}{\hat{e}}
\Mdef{\hf}{\hat{f}}
\Mdef{\hA}{\hat{A}}
\Mdef{\hH}{\hat{H}}
\Mdef{\hJ}{\hat{J}}
\Mdef{\hM}{\hat{M}}
\Mdef{\hP}{\hat{P}}
\Mdef{\hQ}{\hat{Q}}

\Mdef{\thetab}{\overline{\theta}}
\Mdef{\phib}{\overline{\phi}}

\Mdef{\uA}{\underline{A}}
\Mdef{\uB}{\underline{B}}
\Mdef{\uC}{\underline{C}}
\Mdef{\uD}{\underline{D}}

\Mdef{\bolda}{\mathbf{a}}
\Mdef{\boldb}{\mathbf{b}}
\Mdef{\bfD}{\mathbf{D}}

\Mdef{\fm}{\frak{m}}
\Mdef{\fp}{\frak{p}}

\Mdef{\eps}{\epsilon}

\input{xypic}
\newcommand{\connsub}{\mathrm{ConnSub}}

\newcommand{\flag}{\mbox{flag}}
\newcommand{\cEi}{\cE^{-1}}
\newcommand{\mbd}{\mathbf{d}}
\newcommand{\sub}{\mathrm{Sub}}
\newcommand{\ist}{i_{\sigma}^{\tau}}

\newcommand{\elr}[1]{E\langle #1 \rangle}
\newcommand{\elrG}[1]{E_G\langle #1 \rangle}
\newcommand{\elrN}[1]{E_{\mN}\langle #1 \rangle}
\newcommand{\elrT}[1]{E_{\mT}\langle #1 \rangle}
\newcommand{\elrP}[1]{E_{P}\langle #1 \rangle}
\newcommand{\siftyV}[1]{S^{\infty V(#1)}}
\newcommand{\GI}{\mathcal{GI}}

\newcommand{\Sigmat}{\widetilde{\Sigma}}

\newcommand{\qp}{P}
\newcommand{\pes}{\pi^e_{!}}
\newcommand{\RR}{{\mathbb{R}}}
\newcommand{\RRa}{{\mathbb{R}}_a}
\newcommand{\RRap}{{\mathbb{R}}_a^p}
\newcommand{\RRtw}{{\mathbb{R}}_{tw}}
\newcommand{\RRinv}{{\mathbb{R}}_{inv}}
\newcommand{\RRc}{{\mathbb{R}}_c}
\newcommand{\RRcb}{{\overline{\mathbb{R}}}_c}
\newcommand{\RRd}{{\mathbb{R}}_d}

\newcommand{\TC}{\mathcal{TC}}

\newcommand{\cAp}{\cA^p}
\newcommand{\cAf}{\cA^f}
\newcommand{\mT}{\mathbb T}
\newcommand{\mN}{\mathbb N}
\newcommand{\mW}{\mathbb W}

\newcommand{\adjointtriple}[5]{
\diagram
#1 \ar[rr]|-{#3} &&
#5  
\llto<1.2ex>^{#4}
\llto<-1.2ex>_{#2} 
\enddiagram}
\newcommand{\toral}{\Lambda (\mT)}
\newcommand{\etoralp}{E\toral_+}
\newcommand{\piAG}{\pi^{\cA (G)}}
\newcommand{\piAT}{\pi^{\cA (\mT)}}
\newcommand{\piAN}{\pi^{\cA (\mN)}}

\newcommand{\bbI}{\mathbb{I}}
\newcommand{\fV}{\mathfrak{V}}
\newcommand{\fW}{\mathfrak{W}}
\newcommand{\fG}{\mathfrak{G}}
\newcommand{\lr}[1]{\langle #1\rangle }
\newcommand{\suppcod}{\mathrm{scd}}
\begin{document}
\title{Rational equivariant cohomology theories with toral support}
\author{J.P.C.Greenlees}
\address{School of Mathematics and Statistics, Hicks Building, 
Sheffield S3 7RH. UK.}
\email{j.greenlees@sheffield.ac.uk}
\date{}

\begin{abstract}
For an arbitrary compact Lie group $G$, we describe a model for
rational $G$-spectra  with toral geometric isotropy and show that
there is a convergent Adams spectral sequence based on it. 
The contribution from geometric isotropy at a subgroup
$K$ of the maximal torus of $G$ is captured by a module over
$H^*(BW_G^e(K))$ with an action of $\pi_0(W_G(K))$, where $W_G^e(K)$
is the identity component of $W_G(K)=N_G(K)/K$.
\end{abstract}

\thanks{I am grateful to MSRI for support and providing an excellent environment
  for organizing these ideas during the Algebraic Topology Programme
  in 2014}
\maketitle

\tableofcontents
\section{Introduction}
\subsection{Main result}
For any compact Lie group $G$, rational $G$-equivariant cohomology theories are represented by
rational $G$-spectra. Furthermore, the category of $G$-equivariant cohomology
theories is the homotopy category of the category of rational
$G$-spectra. This category breaks up into mutually orthogonal parts,
the most important of which is the toral part: the cohomology theories
are those with toral support, and the $G$-spectra are those 
whose geometric isotropy is a set of subgroups of the maximal torus
$\mT$.

 In this paper we provide an effective method for calculating with toral $G$-spectra. 
 More precisely, we construct an
abelian category $\cA (G,toral)$ of injective dimension equal to the
rank of $G$ and a homology functor 
$$\piAG_*: \mbox{$G$-spectra} \lra \cA (G,toral), $$
so that (Theorem \ref{thm:ASS}) there is an Adams spectral sequence
$$\mathrm{Ext}^{*,*}_{\cA (G,toral)}(\piAG_*(X),
\piAG_*(Y))\Rightarrow [X,Y]^G_*$$
convergent for arbitrary rational toral $G$-spectra $X$ and $Y$. The
special cases when $G$ is a torus, $O(2)$ and $SO(3)$ follow from
earlier work \cite{s1q,o2q,so3q}. 

In all cases, the model is assembled from data at individual subgroups
$K$. The contribution from $K$ comes from the geometric $K$-fixed
point spectrum; this spectrum has an action of the Weyl group
$W_G(K)=N_G(K)/K$, with identity component $W_G^e(K)$ and discrete
quotient $W_G^d(K)=\pi_0(W_G(K))$.  When the spectrum is finite  the piece of data 
amounts to taking the $W_G^e(K)$-equivariant Borel cohomology of its
dual and viewing it as a module over $H^*(BW_G^e(K))$ with an action
of $W_G^d(K)$ (see Proposition \ref{prop:piAG} for a complete statement).

\subsection{Background}
If $G$ is a compact Lie group, the author has conjectured
\cite{gqsurvey} that one may  describe the homotopy theory of
rational $G$-spectra in algebraic terms. There are now a good
number of examples where this has been proved, including finite groups, tori
\cite{tnqcore}, $O(2)$ \cite{Barnes1, Barnes2, Barnes3} and
$SO(3)$ \cite{Kedziorek}. The results show that there is a Quillen equivalence between
the category of rational $G$-spectra and differential graded objects
in a certain abelian category $\cA (G)$. 

The category $\cA (G)$ takes
the form of a category of sheaves of modules over a sheaf of rings on the space of
closed subgroups of $G$: the stalk over  a subgroup $H$ captures 
the geometric isotropy information at $H$. Many of the most
interesting cohomology theories (including $K$-theory and elliptic
cohomology) have the property that the geometric isotropy comes
entirely from subgroups of the maximal torus.  For example, it is
apparent from the groups $O(2)$  and $SO(3)$  that the part of the model
corresponding to isotropy in the maximal torus $\mT$ is the most
significant and interesting part. The present paper is about this
toral part for an arbitrary compact Lie group $G$.

To be more precise, the endomorphism ring of the
rational sphere spectrum (the rational Burnside ring $A(G)$) acts on the
category of $G$-spectra, and in fact it consists of the equivariant sections of
the constant sheaf $\Q$ over the space $\cF G$ of subgroups of finite index in
their normalizer. This means
 that  $A(G)=C_G (\cF G, \Q)$ is the ring of equivariant continuous
 functions, where  $\cF G$ has the Hausdorff metric topology.  Any open, closed, $G$-invariant subset  $S $ of
$\cF G$ specifies an idempotent $e_S\in A(G)$, and the category of rational
$G$-spectra and $\cA (G)$ both split into two pieces corresponding to
the decomposition $1=e_S+e_{S^c}$. The part corresponding to $e_S$
consists of spectra whose geometric isotropy consists of subgroups $L$
cotoral in elements of $S$ (i.e., $L$ is normal in a subgroup $K$ in 
$S$ with $K/L$ a torus). In particular, we may take $S$ to consist
of the single conjugacy class $(\mT)$ of maximal tori, and consider the category 
$$\mbox{toral-$G$-spectra}:=e_{(\mT)} \left[\mbox{$G$-spectra/$\Q$} \right] $$
 consisting of $G$-spectra whose geometric isotropy lies inside a
 maximal torus. 

In general we may break up the category of rational $G$-spectra by
choosing a finite orthogonal decomposition of the unit of $A(G)$. For
tori $\cA (G)$ is indecomposable and $\cA (G)=\cA (G,toral)$. For $G=O(2)$ the category breaks up
into the toral (or cyclic) part $\cA (O(2), toral) =\cA (SO(2))[W]$ as described here, and a piece
corresponding to dihedral groups which is simply a graded equivariant
sheaf over a compact totally disconnected space  with $O(2)$ as an
accumulation point \cite{o2q, Barnes1, Barnes2, Barnes3}.  For $G=SO(3)$ the category again breaks up
into the toral (or cyclic) part $\cA (SO(3), toral)$ as described here, and a piece
which is simply a graded equivariant sheaf \cite{so3q, Kedziorek}; the
graded sheaf piece also breaks up into a piece 
corresponding to dihedral groups (of order 4 or more) which have
$O(2)$ as an accumulation point, and a piece for 
a number of isolated exceptional subgroups (tetrahedral, octahedral and
icosohedral).  One should not expect that the non-toral part is always
a plain graded sheaf; for example, if $G$ is the product of a circle $T$ and a group of order 2
splits into a the toral part $\cA (G,toral)$ (as here) and a second part
(corresponding to subgroups not inside the maximal torus) which is
similar in character to $\cA (T)$. The category of toral chains
 described in \cite{ratmack} gives an indication of the expected
 pattern in general. 

\subsection{Restriction to the maximal torus}

In a pattern familiar from elsewhere in the theory of transformation groups,
 we  will prove that  restricting from $G$-spectra to  $\mT$-spectra
is faithful  on the homotopy category of toral $G$-spectra provided we
remember the action of the Weyl group $\mW G$.  This  suggests that the putative algebraic model $\cA
(G,toral)=e_{(\mT)}\cA (G)$ for toral $G$-spectra could be described
in terms of the category $\cA (\mT)$ defined in \cite{tnq1}. This idea
will lead us to the construction of a category  $\cA (G,toral)$ and a homology theory on $G$-spectra with
values in $\cA (G, toral)$. We will show that this is a good invariant
in the sense that it is calculable and gives a convergent Adams spectral sequence for maps between toral spectra. 

Consideration of the torus-normalizer $\mN =N_G(\mT)$ is central to
the analysis. The general
theory simplifies for $\mN$, since the identity component is itself the maximal
torus, and  we find $\cA (\mN, toral)=\cA (\mT )[\mW G]$ (the category
of $\mW G$-equivariant objects of $\cA (\mT)$ in a sense  made precise
below).  We show that restriction from $G$-spectra to $\mN$-spectra is full
and  faithful on homotopy categories. In summary, we construct a diagram 
$$\diagram
\mbox{toral-$G$-spectra} \rto^{\res^G_{\mN}} \dto_{\piAG_*} & \mbox{toral-$\mN$-spectra} \rto^{\res^{\mN}_{\mT}} \dto_{\piAN_*} & 
\mbox{$\mT$-spectra} \dto_{\piAT_*}\\
 \cA (G,toral) \rto & \cA (\mN,toral) \rto \ar@{=}[d]&
\cA (\mT)\\
&\cA (\mT )[\mW G]&
\enddiagram$$
with convergent Adams spectral sequences based on each of the vertical homology
functors. 
Taken together with the examples mentioned above, this is strong evidence for the conjecture that the category of toral
$G$-spectra is Quillen equivalent to the category of differential graded objects in
$\cA (G, toral)$. 

\subsection{The form of the model of toral $G$-spectra}
As suggested by the known examples,  we expect the stalk of $\cA
(G,toral)$ over a subgroup $K$ to be a module over $H^*(BW_G^eK)$, where
$W_G^e(K)$ is the identity component of the Weyl group $W_GK=N_GK/K$,
and we expect an action of the discrete quotient $W_G^dK=\pi_0(W_GK)$. 
More precisely, we expect a module over the twisted group ring 
$$\RRtw^G (K)=H^*(BW_G^eK) [W_G^dK].  $$
Understanding the restriction from $G$-spectra to $\mN$-spectra involves some rather
interesting pieces of invariant theory.

The relationship between the stalks is given by the localization theorem for cotoral
inclusions. If $G=\mT$ is a torus and $L$ is cotoral in $K$ then we have a group homomorphism 
$W_GL=\mT /L \lra \mT /K =W_GK$ which forms the basis of this.  For a general
group $G$ we cannot expect a map  $N_GL \lra N_GK$  (consider $L=1$), so we think of rings and modules associated
to cotoral {\em flags} of subgroups, and this restores functoriality.  We  recall how this works for tori
in the next subsection.

\subsection{The model for $\mT$-spectra}
\label{subsec:modelAT}
We sketch the construction of the model $\cA (\mT)$ for rational
$\mT$-spectra, referring to \cite{AGs} for fuller details (the model
described  here is the $(a,f)$-model, based on  {\bf
  f}lags involving {\bf a}ll closed subgroups). The starting point
of the discussion is the poset $\Sigma_a$ consisting of {\bf a}ll closed
subgroups. The partial order is given by cotoral inclusion, so that 
$K\supseteq L$ if $L$ is a subgroup of $K$ and $K/L$ is a torus. 
We then consider the poset $\flag (\Sigma_a)$ consisting of
flags $(K_0\supset K_1\supset \cdots \supset K_s)$ in $\Sigma_a$.
We may define a $\Sigma_a$-diagram $\RRa $ of rings by 
$$\RRa (K):=H^*(B\mT /K). $$
If $K\supseteq L$ then the projection map $\mT /K\lla \mT /L$ induces the
inflation map $\RRa (K)\lra \RRa (L)$, making $\RRa$ into a
contravariant functor on $\Sigma_a$.

Using Euler classes, we may form a $\flag (\Sigma_a)$-diagram of
rings. Since the values on flags of length 0 agree with those of $\RRa$, we
continue to use $\RRa$ for the functor on flags. First, if $K\supseteq L$ we may  consider the set 
$$\cE_{K/L}:=\{ e(W)\st W\in \mathrm{Rep} (\mT /L), W^K=0\}$$
of Euler classes of $K$-essential representations of $\mT /L$. Here 
$e(W)\in H^{|W|}(B\mT /L)$ is the Euler class of $W$. Now we may define the flag diagram $\RRa$ by 
$$\RRa (K_0\supset K_1\supset \cdots \supset
K_s):=\cEi_{K_0/K_s}H^*(B\mT /K_s).$$
We note that this only depends on the first and last term in the
flag and it is a {\em covariant} functor on $\flag (\Sigma_a)$.
 It is also important to note that if $K\supset L$ the values at
$K$ and $L$ are not linked directly in $\flag (\Sigma_a)$, but rather
through the zig-zag induced by the inclusions $(K)\lra (K\supset L)\lla (L)$. 

The category $\cA (\mT)$ is a category of modules $M$ over the $\flag
(\Sigma_a)$-diagram $\RRa$: thus $M$ is a diagram of abelian groups so
that if $E\subseteq F$, the map $M(E)\lra M(F)$ is a map of modules
over $R(E)\lra R(F)$. The modules in $\cA (\mT)$ are required to be quasicoherent
($qc$), extended ($e$) and $\cF$-continuous. A module is {\em quasi-coherent} if    
 the value is determined by the last term in the flag
by extensions of scalars: if $F=(K_0\supset K_1\supset \cdots \supset
K_s) $
then the inclusion $(K_s)\lra F$ induces an isomorphism 
$$M(F) =\RRa (F) \tensor_{\RRa (K_s)}M(K_s). $$
A module is {\em extended} if 
 the value is determined by the first term in the flag
by extensions of scalars:
 if $F=(K_0\supset K_1\supset \cdots \supset
K_s) $
then the inclusion $(K_0)\lra F$ induces an isomorphism 
$$M (F) =\RRa (F) \tensor_{\RRa
 (K_0)}M(K_0). $$
The $\cF$-continuity condition is a finiteness condition described in
Subsection \ref{subsec:qceG} below.

Since the values of both $\RRa$ and $M$ are determined by the first
and last term in any flag we will sometimes simplify the notation by
giving the value only on cotoral pairs $(K\supset L)$. The point of
defining $\RRa$ on flags is to give functoriality and control automorphisms.

Our approach to constructing $\cA (G,toral)$  is to take into account
the action of the Weyl group  $\mW G=N_G(\mT)/\mT$ on the poset $\Sigma_a(\mT )$ of 
all subgroups of the maximal torus. Indeed,  $\mW G$ acts on the
diagram $\RRa$ of polynomial rings, and it turns out that by descent
this gives us the model $\cA (G,toral)$. We develop the appropriate
machinery, and finally give the definition in Subsection \ref{subsec:model}

\subsection{Layout of the paper}
The paper is divided into two parts.  The first (``Algebra'') develops the algebraic
framework and second (``Topology'') applies it to calculations with $G$-spectra.

In Section \ref{sec:Lie} we introduce notation from the theory of compact Lie groups and make some
elementary observations, and in Section
\ref{sec:HBG} we recall facts about the cohomology of classifying
spaces of compact Lie groups. 

We then spend several sections developing machinery to discuss
categories of modules over a diagram of rings on which a finite group
acts. In Section \ref{sec:equivdiagrams} the categorical setup is
described and in Section  \ref{sec:equivringdiagrams} this is
specialised to categories of modules over an equivariant diagram of
rings and the fundamental descent
adjunction is proved. Finally, Section \ref{sec:Liediagrams}
specializes to the example arising from compact Lie groups, and gives
the definition of $\cA (G,toral)$. The fact that the descent
adjunction respects quasi-coherent extended modules contains
information on which  $\mN$-equivariant objects are restrictions of
$G$-equivariant objects; this is surprisingly subtle and
treated in Section \ref{sec:GspNsp}. As a first step towards homotopy theory,
we then consider the homological algebra of $\cA (G,toral)$,
identifying enough injectives and showing its injective dimension is equal to
the rank of $G$.

We then turn to topology. The fundamental result proved in Section
\ref{sec:toraldetect}
is that toral phenomena are detected on restriction to the maximal
torus. In preparation for work on the Adams spectral sequence we then 
reformulate some well known properties of Borel cohomology in Section 
\ref{sec:Borel}; this is the route by which the classical theory of
characteristic classes of principal bundles enters the model. 

Section \ref{sec:GspectratoAG} explains the relationship between $\cA
(\mT)$ and $\cA (G,toral)$ and thereby allows us to construct the functor
$\piAG_*$ from $G$-spectra to $\cA (G,toral)$. This is then used in
Section \ref{sec:ASS} to construct the Adams spectral
sequence, with the hard work deferred to Section \ref{sec:objects}
where enough injectives are realized, and Section \ref{sec:mapstoinj} 
where it is shown that maps into
the resulting spectra are detected in $\cA (G,toral)$. 
The Adams spectral sequence lets one calculate maps, and
this is complemented in Section \ref{sec:essepi} by a proof that the
functor $\piA_*$ is essentially surjective: all objects of $\cA
(G,toral)$ do occur as $\piAG_*(X)$ for a toral $G$-spectrum $X$. 
Finally, Section \ref{sec:hg} explains how restriction, induction and
coinduction are reflected at the level of algebraic models. 

\subsection{Conventions}

All groups will be compact Lie groups, and if connectedness is
required this will be stated explicitly. All subgroups will be
required to be closed. Generally, containment of subgroups 
follows the alphabet, as in $G\supseteq H$.

Cohomology is unreduced unless indicated, and always has rational coefficients.

\part{Algebra}
\section{Weyl groups}
\label{sec:Lie}

The algebraic input to our results is the classical structure and
representation theory of compact Lie groups. Although this is
well-known, the recollection of standard facts gives an opportunity 
to  introduce notation. Readers have found the list of standard
notation  in Subsection \ref{subsec:notn} valuable. We recommend 
consideration of the rotation group $G=SO(3)$ as an example
to illustrate results.

\subsection{Two types of Weyl groups}
For a compact Lie group $G$ we write
$G_e$ for its identity component and $G_d=G/G_e=\pi_0(G)$ for
the discrete quotient. A closed subgroup $K$ of $G$  has 
 normalizer $NK=N_GK$, and  Weyl group $WK=W_GK=N_GK/K$. 
More generally given a flag $F=(K_0\supset K_1\supset \cdots \supset K_s)$ of
subgroups of $G$ we write
$$N_G(F)=N_G(K_0)\cap \ldots \cap N_G(K_s)$$
for the subgroup normalizing all terms in the flag and 
$$W_G(F)=N_G(F)/K_s$$
for its Weyl group. 

Moving on to the theory of compact Lie groups,  we write 
$\mT =\mT G$ for the maximal torus of $G$,  $\mN =\mN G=N_G(\mT)$ for
its normalizer. The  Weyl group $W_G(\mT)=N_G(\mT)/\mT$  of the
maximal torus is a finite group that plays a central role in the
theory, so we use the notation
$$\mW G = W_G\mT . $$
Since $W_G(G)=1$, there is little danger of confusion as long as the
reader bears in mind there are two meanings of the phrase `Weyl
group'. 

For most of this paper we  will suppose $K$ is a  subgroup of $\mT$, so that $\mT
\subseteq N_GK$.

\subsection{Weyl groups of Weyl groups}
We will need to consider $WK$ as a Lie group in its own right,
with maximal torus $\mT WK$ and Weyl group $\mW WK=W_{WK}(\mT
WK)=N_{WK}(\mT WK)/\mT WK$. We may simplify this notation slightly. 

\begin{lemma}
The maximal torus of $WK$ is given by 
$\mT WK=\mT/K$. 
\end{lemma}

\begin{proof}
Certainly $\mT/K$ is a torus in $WK$. If there were a bigger one it
would have the form $T'/K$ for some subgroup $T'$ of $NK$ containing
$\mT$. Then we would have a chain $T'\supseteq \mT\supseteq K$. The
group  $T'/\mT$ is a quotient of the torus $T'/K$ and hence is itself a torus. Thus $T'$
is itself a torus; by maximality of $\mT$ we have $\mT=T'$. 
\end{proof}

\begin{lemma}
\label{lem:WGK}
The normalizer of the maximal torus of $WK$ is given by 
$$N_{WK}(\mT WK)=N_{WK}(\mT/K)=N_G(\mT\supseteq K)/K=(N_G\mT\cap N_GK)/K.$$
It follows that the toral Weyl group of $W_GK$, 
$$\mW W_GK= (N_G\mT \cap N_GK)/\mT\subseteq N_G\mT/\mT=\mW G$$
is the subgroup of $\mW G$ normalizing $K$. With the usual notation
for the isotropy group of the action of $\mW G$ on the set of
subgroups of $\mT$ we have 
$$\mW (W_GK)=(\mW G)_K.$$ 
\end{lemma}

\begin{proof}
The first equality is the previous lemma. Now any element $g$ of $G$
normalizing $\mT\supseteq K$ is in $NK$ and hence defines an element
$gK$  of
$WK$. We then have $(gK)(tK)(gK)^{-1}=gtg^{-1}K$. This is in $\mT$ by
hypothesis, and  hence we have
a homomorphism 
$$N_G(\mT\supseteq K)\lra N_{WK}(\mT/K). $$
Evidently $K$ is in the kernel, and since $N_{WK}(\mT/K)\subseteq WK$ we
have a monomorphism 
$$N_G(\mT\supseteq K)/K \lra N_{WK}(\mT/K) . $$
Now suppse $gK \in WK$ normalizes $\mT/K$, which is to say that for any
$t\in \mT$, 
$$\mT/K\ni (gK)(tK)(gK)^{-1}=gtg^{-1}K$$
It follows that $gtg^{-1} \in \mT$ and $g$ normalizes $\mT$. 
\end{proof}

\subsection{Summary of Notation}
\label{subsec:notn}
Associated to a subgroup $K$ of $\mT$  we have

\begin{itemize}
\item $G$, the ambient compact Lie group
\item $G_e$ the identity component of $G$
\item $G_d=G/G_e=\pi_0(G)$ the group of components of $G$
\item $\mT = \mT G$ the maximal torus of $G$
\item $\mN = \mN G=N_G(\mT)$
\item $\mW = \mW G=\mN G/\mT G$ the toral Weyl group of $G$
\item $K$ a closed subgroup of $\mT G$
\item $NK=N_GK$
\item $WK=W_GK$; with  identity component $W^eK=W_G^eK=(W_GK)_e$ and
 component group $W^dK=W_G^dK=(W_GK)_d$
\item $\mT WK=\mT/K$
\item $\mN WK =N_{WK}(\mT WK)=(N\mT\cap NK)/K$
\item $\mW WK =(N\mT \cap NK)/\mT =(\mW G)_K$ acting on $\mT WK=\mT /K$
\end{itemize}

\section{Cohomology of classifying spaces of compact Lie groups}
\label{sec:HBG}
The relationship between the rational cohomology of classifying
spaces of $G$, $\mN$ and $\mT$  proved by Borel is fundamental to our
entire analysis. 

\subsection{Cohomology of classifying spaces and free spectra}

Borel's  calculation of the rational cohomology of classifying spaces
is as follows. 
\begin{lemma} (Borel)
For a compact Lie group $G$ with maximal torus $\mT$,  $\mN =N_G(\mT)$
and $\mW =N_G(\mT)/\mT$ we have 
$$H^*(BG)=H^*(B\mN)=H^*(B\mT )^{\mW}.\qqed $$
\end{lemma}

We may apply this to Weyl groups to see 
$$H^*(BWK)=H^*(B\mN WK) =H^*(B\mT WK )^{\mW WK} =H^*(B\mT /K )^{\mW WK} .$$

Cohomology of classifying spaces plays a fundamental role in
equivariant stable homtopy theory. 
\begin{thm} (Greenlees-Shipley \cite{gfreeq2}) The category of  free rational $W_GK$-spectra
 is Quillen equivalent to the category of  torsion modules over the twisted group ring
$$H^*(BW_G^eK)[W_G^dK] =
H^*(B\mT /K)^{\mW W_G^eK}[W_G^dK].\qqed $$
\end{thm}

This embodies the role of the cohomology of classifying spaces in
modelling rational stable equivariant homotopy theory.

\subsection{Identity components 1}
The maximal torus only depends on the identity component of a group,
so  $\mT (G_e)=\mT G$. 

\begin{lemma}
\label{lem:NGNGe}
There are short exact sequences
$$1\lra N_{G_e}\mT \lra N_{G}\mT \lra G_d\lra 1$$
$$1\lra \mW G_e \lra \mW G  \lra G_d\lra 1$$
\end{lemma}
\begin{proof} $G_e$ acts transitively on maximal tori. So for any
  $\gamma \in G_d$ represented by $\tilde{\gamma}\in G$ there is an $x \in
  G_e$ with $\mT=(\mT^{\tilde{\gamma}})^x$ and $\tilde{\gamma}x \in
  N_G(\mT)$. Since $N_{G_e} (\mT)=N_G(\mT)\cap G_e$, this gives the
  exact sequence. 
\end{proof}

This fits well with the following picture
$$\diagram 
G=G_e\cdot G_d&
H^*(BG)\dto_{\cong}\rto^{\cong} &H^*(BG_e)^{G_d}\dto^{\cong} \\
N_G\mT =N_{G_e}\mT \cdot G_d& 
H^*(BN_G\mT )\dto_{\cong}\rto^{\cong}
&H^*(BN_{G_e}\mT)^{G_d}\dto^{\cong} \\
W_G\mT =W_{G_e}\mT \cdot G_d& 
H^*(B\mT )^{\mW G}\rto^{\cong} &(H^*(B\mT)^{\mW G_e})^{G_d}
\enddiagram$$

\subsection{Identity components 2}

 We note that $\mN$ acts on the set subgroups of $\mT$
by conjugation, and that this passes to an action of $\mW$. 
Recall that
$$\mW W_GK =(N_GK \cap \mN)/\mT=(\mW G)_K.$$ 

\begin{lemma}
There is a short exact sequence
$$1 \lra \mW ( W_G^eK) \lra \mW  (W_GK) \lra W_G^dK\lra 1. $$
Under the identification $\mW (W_GK)=(\mW G)_K$, the subgroup $\mW (W_G^eK)$
corresponds to  the set of elements of the toral Weyl
group $\mW G$ represented by the identity component of $N_G(K)$. 
$$(K\cdot (N_GK)_e \cap \mN)/\mT = \mW (W_G^eK)$$
\end{lemma}
\begin{remark}
The Weyl group $\mW W_G^eK$ can be very small or very large. At one
extreme, if $K=\mT$ it is trivial. At the other,  if $K=1$ then $N_G(K)=G$ and if $G$ is connected we
obtain the entire Weyl group $\mW G$. 
\end{remark}
\begin{proof}
The short exact sequence is obtained by applying Lemma \ref{lem:NGNGe} to
$WK$. Now note  $\mW (W_G^eK)=N_{W_G^eK}(\mT /K)/(\mT /K)$ and 
$W_G^eK=K\cdot (N_GK)_e /K$. 
\end{proof}

\section{Equivariant diagrams}
\label{sec:equivdiagrams}

We are going to discuss diagrams of rings and modules with group
action. The basic examples arise from the algebraic models of rational
$\mT$-spectra \cite{AGs}, recalled briefly in Subsection
\ref{subsec:modelAT}.  The diagram shapes $\Sigma$
come from the set $\Sigma_a (\mT)$ of all closed subgroups of $\mT$
under cotoral inclusion (in the present context, simply inclusions with
connected quotient).  One such poset $\Sigma$  is $\Sigma_a$ itself, but we also need to consider
the poset $\flag (\Sigma_a)$ of flags in  $\Sigma_a$. Accordingly, we 
discuss the relevant structures with $\Sigma $ unspecified, which has
the added benefit of clarifying the structure.

\subsection{Diagrams with an action}
We need to consider the general setup of a group $W$ acting on
the right of a poset $\Sigma$. We want a notion of equivariant $\Sigma$-diagrams in a
category $\C$. We start by considering the  functor category
$\C^{\Sigma}$. This admits
an action of $W$, where the image of a functor $F: \Sigma \lra \C$
under $w\in W$ is the functor $w_*F$ defined by 
$$(w_*F)(\sigma) :=F(\sigma^w). $$
One quickly verifies  $v_*w_*F=(vw)_*F$ and $e_*F=F$. 

An equivariant diagram is then a diagram $F$ with additional
structure. We specify an action by $W$ on $F$ by giving maps 
$$w_m: F \lra w_*F$$
with $e_m=1$ and $v_mw_m=(vw)_m$. It is more flexible to give an
alternative point of view in which an equivariant $\Sigma$-diagram is
just a diagram of a more complicated shape. 

\subsection{Orbifold posets}
We want to consider a class of categories $A$ that are based on a poset
$\Sigma$, but with automorphisms added. 

\begin{defn}
A {\em $\Sigma$-orbifold}  is a category $A$ with the same objects as
$\Sigma$ equipped with functors  $\Sigma\lra A \lra \Sigma$ which are the
identity on objects so that 
\begin{enumerate}

\item $A$ has finitely many morphisms between any two objects, 
\item the morphisms of $A$ are generated by those of $\Sigma$ together with the
  automorphisms, and 
\item every endomorphism in $A$ is an isomorphism
\end{enumerate}
\end{defn}


The {\em trivial} $\Sigma$-orbifold associated to the finite group $W$
is $A=\Sigma \times W$ with structure maps coming from $1\lra W\lra
1$. 


\subsection{The transport category}
Starting with a poset $\Sigma$ with a right action of a finite group $W$, 
we may form the transport category $\Sigma \sdr W$. 
This is a $\Sigma$-orbifold with morphism set $\Sigma \times
W$ and structure maps induced by $1\lra W \lra 1$.  
In giving formulae for composition we are following through the convention that
the action of $W$ is on the right, so that functions also operate on
the right.  

If $i: \sigma \lra \tau$ and $v\in W$ the morphism $(i,v)$ has domain
$\sigma$ and codomain $\tau^v$.  The  composite of $(i,v)$ and $(j,w)$
where $j: \tau^v \lra \phi^v$ is given by
the formula
$$(i,v)(j,w)=(ij^{v^{-1}}, vw). $$
Note that $(i,v)$ is the composite $\sigma
\stackrel{i}\lra \tau \stackrel{v}\lra \tau^v$. Since $W$ acts on
$\Sigma$ as a poset, we may find a commutative square
$$\diagram 
\tau \rto^v& \tau^v\\
\sigma\rto^v \uto^{i}&\tau^v \uto_{i^v}
\enddiagram $$
and
$$i v =(i,v)=v i^v. $$
In particular the self-maps of $\sigma$ as an object of $\Sigma \sdr
W$ is the isotropy group $W_{\sigma}$. 

It is clear we can repackage the notion of a $W$-equivariant diagram
in terms of $\Sigma \sdr W$. 

\begin{lemma}
 \label{lem:Wequivfunctorial}.
The category of  $W$-equivariant $\Sigma$-diagrams in $\C$ is equivalent to 
the category of  functors $ \Sigma \sdr W
\lra \C$. 
\end{lemma}

\begin{proof}
Equivariant diagrams $(F,\{ w_m\}_{w\in W})$  are related to  functors $F':\Sigma \sdr W
\lra W$ by taking $F'(\sigma)=F(\sigma)$ and $F'(w)=w_m$. 
\end{proof}

\subsection{Component structures}
The purpose of the formulation in terms of the transport category
$\Sigma \sdr W$ is to let us to capture the behaviour of  the identity
components of  Weyl groups.

\begin{defn}
A {\em component structure} on $\Sigma \sdr W$ is a sub-$\Sigma$-orbifold
$W_{\bullet,}^e$. Given a component structure, the endomorphism
object of $\sigma$ is written $W_{\sigma}^e$. 

A component structure is {\em normal} if $W_{\sigma }^e$ is normal in
$W_{\sigma}$.
\end{defn}
\begin{lemma}
 If the component structure is normal the {\em discrete residual} is
 the $\Sigma$-orbifold $W_{\bullet }^d$ with a
sequence maps
$$W_{\bullet }^e\lra  \Sigma \sdr W\lra  W_{\bullet }^d$$
of $\Sigma$ orbifolds which is exact in the sense
that it defines an isomorphism  $W_{\sigma}^d\cong W_{\sigma}/W_{\sigma }^e$.
\end{lemma}

\begin{proof}
The morphisms in $W_{\bullet}^{d}$ are pairs $(i,[v])$ where $i: \sigma
\lra \tau$ and were
$[v]$ is the equivalence class of $v\in W$ under precomposition by
$W_{\tau}^{e}$ and postcomposition by $W_{\tau^v}^{e}$. The composition is induced
from the composition of $\Sigma \sdr W$. The normality condition
enables one to check that this is well defined. 
\end{proof}

\begin{example}
(i) For any $W$ we may define the {\em connected}  component structure by  
$W_{\sigma}^{e}=W_{\sigma}$ giving $W_{\sigma}^{d}=1$.

(ii) For any $W$ we may define the {\em discrete}  component structure by  
$W_{\sigma}^{e}=1$ giving $W_{\sigma}^{d}=W_{\sigma}$.
\end{example}
We devote a separate subsection to the motivating example that will
concern us for most of the paper.
\subsection{The compact Lie group component structure}
\label{subsec:Liecomponent}
The motivating example comes from a compact Lie group $G$. We take
$\Sigma =\Sigma_a(\mT) $ and the Weyl group  $W=\mW G$ acts by
conjugation in the usual way. 

The  component structure corresponds to the identity components of the 
Weyl groups 
$$W_{K}^e=(\mW G)_{K}^{e}=\mW W_G^eK$$ 
Accordingly, by Lemma \ref{lem:NGNGe},  the discrete residual is 
$$W_{K}^{d} =(\mW G)_K^d=W_G^dK.$$

\begin{example}
If the identity component of $G$ is the maximal torus
$\mT$, so that $G=\mN$ we have $W_G^eK=\mT /K$ which has trivial toral
Weyl group, and the component structure is the discrete component structure
$$(\mW \mN)_K^e=1 \mbox{ and }(\mW \mN)_K^d=(\mW \mN)_K.$$
\end{example}

\begin{example}
If $G=SO(3)$ we have $\mT =SO(2)$, $\mN =O(2)$ and $\mW =C_2$. 
The subgroups of $\mT$ are the cyclic subgroups $C_n$ of finite order $n$, and
$\mT$ itself. All these subgroups are characteristic and hence 
$$(\mW G)_K= \mW G \mbox{ for all } K\subseteq \mT. $$

The trivial subgroup $C_1$ has normalizer $G$ and
Weyl group $W_G(C_1)=G$. The other subgroups of $\mT$ have normalizer 
$\mN$. Thus the finite subgroups have Weyl group isomorphic to $O(2)$
and $\mT$ has Weyl group $\mW G$. The associated component
structure thus has
$$(\mW G)_{K}^e=\dichotomy{\mW G \mbox{ if } K=C_1}
{1 \mbox{ otherwise}} $$
and discrete residual
$$(\mW G)_{K}^d=\dichotomy{1 \mbox{ if } K=C_1}
{\mW G \mbox{ otherwise}} $$
\end{example}

\begin{example}
The group  $G=SU(3)$ has maximal torus $\mT$ of rank 2 consisting of
diagonal matrices. There are three (conjugate) subgroups isomorphic to $SU(2)$ which fix the
first,  second  or third complex coordinate. The Weyl group is the
symmetric group of degree 3 generated by the three corresponding reflections. 

We have 
$$1=(\mW G)_{\mT}^e\subset (\mW G)_{\mT}=\mW G. $$

For subgroups $K\subset \mT$ of dimension 1, we consider the identity 
component  $K_e$. If it is one of the three circles fixed by the
three reflections in $\mW G$ then the normalizer contains the
corresponding $SU(2) $ and 
$$(\mW G)_K^e=(\mW G)_K =\mW SU(2).$$
If $K_e$ is another circle then 
$$(\mW G)_K^e=(\mW G)_K =1. $$

If $K=1$ then $W_G(K)=G$ and 
$$(\mW G)_K^e=(\mW G)_K=\mW G. $$
This is enough to show the richness of the structure; the individual
analysis necessary for the remaining cases can await applications.  
\end{example}

\section{Equivariant diagrams of rings and modules}
\label{sec:equivringdiagrams}
We now specialise the discussion of Section \ref{sec:equivdiagrams} to the case when $\C$ is the
category of commutative rings with a view to establishing the descent
adjunction (Proposition \ref{prop:Psitheta}). 

\subsection{Equivariant diagrams of rings}
Our aim is to describe  a descent theory, relating modules over an
equivariant diagram of rings and modules over the fixed points under
a component structure. We need to impose a restriction on a component 
structure for this to make sense. 

\begin{defn}
We say that a component structure $W_{\bullet}^e$ on the $W$-poset
$\Sigma$ is {\em decreasing} if $E\supseteq F$ implies $W_{E}^e\subseteq
W_{F}^e$.  
\end{defn}

\begin{example}
(i) If $G$ is a connected compact Lie group but not a torus then
$\Sigma_a$ with the Lie group component structure of Subsection
\ref{subsec:Liecomponent}  is not
decreasing: the subgroup $K=1$ has $W_G(1)=G$, with
non-trivial Weyl group $W_{1}^e=\mW G$, whereas the subgroup $K=\mT$ has
discrete Weyl group  $W_G(\mT)$, so that $W_{\mT}^e=1$.

(ii) If $G$ is any compact Lie group then  $\flag (\Sigma_a(\mT))$
with the Lie group component structure of Subsection
\ref{subsec:Liecomponent}
{\em is} decreasing. This is immediate from the fact that 
$$N_G(K_0\supset \cdots \supset K_s)=N_G(K_0)\cap \ldots \cap
N_G(K_s). \qqed$$
\end{example}

The fact that flags give a decreasing structure whereas subgroups do
not explains why we have changed notation for the objects of our
poset. 

\begin{lemma}
Given a $W$-equivariant $\Sigma$-diagram of rings with a decreasing 
component structure $W_{\bullet}^{e}$, the definition 
$$R_{inv} (F)=R(F)^{W_{F}^{e}}$$
gives a $\Sigma$-diagram of rings. 

If the component structure is normal, $R_{inv}$ defines a $W_{\bullet}^{d}$-diagram of rings.
\end{lemma}

\begin{proof}
A map $i: E\lra F$ induces a map $R(i): R(E)\lra R(F)$ and we need to show this
induces a map for $R_{inv}$, namely
$$ R_{inv}(E)=R(E)^{W_{E}^{e}} \lra R(F)^{W_{F}^{e}}=R_{inv}(F)$$ 

The original map $R(i)$ is equivariant
for the inclusion $W_{F}^{e}$. Since the component structure is
decreasing, any $W_{E}^{e}$-invariant element of $R(E)$ is $W_{F}^{e}$
invariant, and so maps to a $W_{F}^{e}$-invariant element of $R(F)$,
and hence $R_{inv}(i)$ is the composite of $R(i)^{W_E^e}$ and
inclusion. 

To see this induces a map on the entire diagram $W_{\bullet}^d$, we
need only observe that the original
structure maps only depend on double coset representatives. 
\end{proof}

An alternative language for describing the resulting structure is
that of {\em twisted}  group rings: if a discrete group $\Gamma$
acts on a ring $R$, the twisted group ring $R[\Gamma]$ has an additive
$R$-basis consisting of the group elements $\gamma \in \Gamma$ and
multiplication is given by $(r\lambda)(s\gamma)=(r
s^{\lambda^{-1}})(\lambda \gamma)$.

\begin{lemma}
The $W_{\bullet}^{d}$ diagram $R_{inv}$
defines the twisted invariant ring
$$R_{tw} (K)=R(K)^{W_{K}^{e}}[W_{K}^{d}] , $$
and this defines a $\Sigma$-diagram of non-commutative rings. 
\end{lemma}

\begin{proof}
Twisted group rings are defined precisely so that the action of the group $W_E^d$ of
endomorphisms of the object $E$ are reflected in a ring acting on the
value at $E$. Since all morphisms are generated by the poset maps and
groups of self-isomporphisms, the twisted group rings give the entire
structure. 
\end{proof}

\subsection{Equivariant diagrams of modules}
The two formulations of $W$-equivariant diagrams of rings have
counterparts for modules. 

\begin{defn}
If $R$ is a $W$-equivariant $\Sigma$-diagram of rings, a
{\em $W$-equivariant module} is an $R$-module which is $W$-equivariant in
the sense that $w_m(\lambda x)=w_m(\lambda) w_m(x)$ for $\lambda \in R$,
$x\in M$ and $w\in W$.
\end{defn}

\begin{lemma}
\label{lem:WequivRmodisdiagram}
The category of $W$-equivariant $R$-modules is equivalent to the
category of  modules over the corresponding $\Sigma \sdr  W$-diagram
of rings. 
\end{lemma}

\begin{proof}
Both $\Sigma$ and $\Sigma \sdr W$ have the same object set. 
The morphism $(i,v): \sigma \lra \tau^v$  is a composite of $(i,e)$
and $(1,v)$. The latter corresponds to the structure map $v_m: M \lra
v_*M$.  The conditions that the actions on rings and modules are compatible
in the two cases correspond to each other. 
\end{proof}

Passing to coset representatives we obtain the result for
$R_{inv}$-modules. 

\begin{lemma}
\label{lem:Wddiagramistw}
The category of  $W_{\bullet}^{d}$-diagrams of $R_{inv}$-modules is
equivalent to the category of  $\Sigma$-diagrams of modules over $R_{tw}$.   \qqed
\end{lemma}
\begin{proof}
Since the conditions that the actions on rings and modules are compatible
in the two cases correspond to each other, this follows by applying the comparison from Lemma
\ref{lem:Wequivfunctorial} to modules.
\end{proof}

\subsection{$R$-modules and $R_{inv}$-modules}
\label{subsec:Psitheta}
A basic technique of equivariant topology is to relate modules over tori to modules over general
groups by suitable descent theorems. We are now equipped to formulate
and prove a  fundamental adjunction which provides the abstract basis
for reducing from $G$-equivariant data to $\mT$-equivariant data.

Suppose that we have a decreasing component structure $W_{\bullet}^{e}$, and 
let $\theta : R_{inv} \lra R$ be the map of $\Sigma$-diagrams
defined by the inclusions $\theta (E): R_{inv} (E)=R(E)^{W_{E}^{e}}\lra R(E)$. 

We define
 $$\Psi =\Psi^{W_{\bullet}^{ e}}: \mbox{$R[W]$-modules}\lra
\mbox{$R_{inv}$-modules}$$
by 
$$(\Psi M)(E):=M(E)^{W_{E}^{e}}.$$
We note that $M(E)$ is an $R(E)$ module, and that passage to fixed
points is lax symmetric monoidal, so that taking fixed points of the
structure maps shows that $M(E)^{W_{E}^{e}}$ is an
$R(E)^{W_{E}^{e}}$-module. Furthermore, if $E\supset F$ then the
structure map $M(F)\lra M(E)$ induces 
$$(\Psi M)(F)=M(F)^{W_{F}^{e}}\lra M(E)^{W_{E}^{e}} =(\Psi M)(E)$$
since $W_{\bullet}^{e}$ is decreasing. 

In the other direction, we may define
$$\theta_*: \mbox{$R_{inv}$-modules}\lra
\mbox{$R$-modules}$$ 
by termwise extension of scalars: 
$$(\theta_*N)(E)=R (E)\tensor_{R_{inv}(E)} N(E)$$
If $E\supset F$ we may define 
$$(\theta_*N)(F)=R(F)\tensor_{R_{inv}(F)} N(F) \lra R(E)\tensor_{R_{inv}(F)} N(E) \lra
R (E)\tensor_{R_{inv}(E)} N(E)=(\theta_*N)(E)$$

\begin{remark} The functor $\theta_*$ is a version of extension of
  scalars for diagrams. On the other hand coextension of scalars does not  give a functor
  of diagrams in general.
\end{remark}


The key result is as follows. 
\begin{prop}
\label{prop:Psitheta}
If $W_{\bullet }^{e}$ is a decreasing normal component structure, and
$R$ is a $W$-equivariant $\Sigma $ diagram of rings,
taking fixed points under a normal component structure $W_{\bullet
  }^{e}$  gives a functor 
$$  \Psi^{W_{\bullet }^{e}}: \mbox{$W$-equivariant-$R$-modules}\lra
\mbox{$R_{inv}$-modules}. $$
This has left adjoint $\theta_*$  given by termwise extension of
scalars.

Provided $|W|$ is invertible in  $R$, the unit map 
$$N(E) \stackrel{\cong}\lra ( R(E)\tensor_{R_{inv}(E)} N(E))^{W_{E}^{e}} = (\Psi
\theta_* N)(E)$$
of the $\theta_*\vdash \Psi$ adjunction is an isomorphism.

\end{prop}


\begin{proof}
For the $\theta_*\vdash \Psi$ adjunction, note that 
objectwise we have
\begin{multline*}
\Hom_{R(E)^{W_{E}^{e}}} (N(E), M(E)^{W_{E}^{e}})  = \Hom_{R(E)^{W_{E}^{e}}}
(N(E), M(E))^{W_{E}^{e}} \\ = \Hom_{R(E)}
(R(E)\tensor_{R(E)^{W_{E}^{e}}}N(E), M(E))^{W_{E}^{e}}
\end{multline*}
so that
$$\Hom_{R_{inv}}(N, \Psi M)=\Hom_R(N, M)^{W_{\bullet, e}}$$
as required.


The unit is an isomorphism since 
$N(E)$ and $R_{inv}(E)$ both have trivial $W_{E}^{e}$-action: one may
take fixed points of the defining coequalizer, and use the fact that
this exact since $|W|$ is invertible. 
\end{proof}

\section{The algebraic model of toral $G$-spectra}
\label{sec:Liediagrams}

We are now ready to specialize the general discussion to the example
arising from compact Lie groups. 
We will describe $\cA (N,toral)$ and $\cA (G, toral)$ using
enrichments of  $\cA (\mT)$.  The starting point is $\Sigma_a (\mT)$
together with  its action of $\mW G$. We will add a little decoration
to indicate which part of the isotropy group of $K$ is internal and
which is external. 

\subsection{Decorating the poset}

We summarize the information we need about a subgroup $K$. These were
discussed in detail in Sections \ref{sec:Lie} and \ref{sec:HBG}

\begin{itemize}
\item $H^*(BW_G^eK)$
\item $W_G^dK$
\item the action of $W_G^dK$ on $H^*(BW_G^eK)$ 
\end{itemize}

Equivalently we need
\begin{itemize}
\item $(\mW G)_K= (\mN \cap NK)/\mT  =\mW WK$
\item The action of $(\mW G)_{K} $ on $\mT /K$
\item The subgroup  $(\mW G)_{K}^{e}$ represented by elements of
  $(N_G(K))_e$
\end{itemize}

Recall  from Lemma \ref{lem:WGK}  that $(\mW G)_{K}^{e} =\mW
(W_G^eK)$, and from Lemma \ref{lem:NGNGe} $(\mW G)_{K}^{d}=W_G^dK$ so that the second and
third pieces of information give
$$H^*(BW_G^eK)=H^*(B\mT /K)^{\mW (W_G^eK)}.$$
The quotient group  $W_G^dK=WK/W_G^eK=(\mW WK)/(\mW (W_G^eK))$ then acts on the
ring of invariants to give the twisted group ring. 

\begin{remark}
We will also need this data with $K$ replaced by a flag $E=(K_0\supset
\cdots \supset K_s)$. This is closely analagous,  once we define
$$N_G(E)=N_G(K_0)\cap \cdots \cap N_G(K_s).$$
Thus $W_G(E)=N_G(E)/K_s$, 
$$(\mW G)_E=(\mN \cap N_G(E))/\mT=(\mW G)_{K_0} \cap \cdots \cap (\mW
G)_{K_s}$$ 
and 
$$(\mW G)_E^e=\mW W_G^e(E).$$
It is clear that this again gives a normal component structure and 
$$(\mW G)_E^d=W_G^d(E).$$
\end{remark}

\subsection{Structures from Lie groups}
The basis of the model is  the diagram $\RRa$ of commutative rings defined on subgroups $K\subseteq \mT$ by
$\RRa (K)=H^*(B\mT /K)$.  For modules $M$ over $\RRa$ there are numerous
structures: $\mW G$ acts on rings, on Euler classes and on
modules. Here we lay out how the structures interact with the group action with
a view to showing it gives examples of  $\mW G$-equivariant
diagrams in the sense of Section \ref{sec:equivdiagrams} above. At
various times  we will consider the poset $\Sigma$ to be either the
poset  $\Sigma_a(\mT )$ of closed subgroups of $\mT$ and cotoral
inclusions or the poset 
$\flag \Sigma_a(\mT)$ of flags in $\Sigma_a (\mT)$ under inclusion.   

The action gives the following structure.
\begin{itemize}
\item Conjugation by an element $w \in \mW G$ gives a group
  homomorphism $r_{w^{-1}}: K^w=w^{-1}Kw\lra K$ and a map $\overline{r}_{w^{-1}}:\mT /K^w \lra \mT /K$. 
\item The conjugation in the previous bullet point gives a ring homomorphism
$$w_m=(\overline{r}_{w^{-1}})^*: \RRa (K)=H^*(B\mT /K)\lra H^*(B\mT /K^w)=\RRa (K^w), $$
with $(vw)_m=v_mw_m$, $e_m=1$.

If we define $w_*\RRa$ by $(w_*\RRa)(K)=\RRa (K^w)$, then we have a
homomorphism of rings $w_m: \RRa \lra w_*\RRa$. This composes by the
rule $(vw)_m=v_mw_m$, $e_m=1$, so that we have an equivariant
$\Sigma_a$-diagram of rings in the sense of Sections
\ref{sec:equivdiagrams}  and \ref{sec:equivringdiagrams}. 
\item Pullback again gives a map $w_m: Rep (G/K)\lra Rep
 (G/K^w)$ on representations. If $U^{H}=0$ then $(w_*U)^{H^w}=0$.
\item By the previous bullet point, given $K\supseteq L$ pullback along the element $w$ maps
 Euler classes $\cE_{H/K}$ to Euler classes $\cE_{H^w/K^w}$.
\item We may therefore define a new diagram $w_*\RRa$ of rings on the poset of
  flags by
 $$(w_*\RRa)(K_0\supset \cdots \supset K_s) =\RRa(K_0^w\supset \cdots
 \supset K_s^w)$$
and we have ring maps $w_m: \RRa \lra w_*\RRa$ satisfying $(vw)_m=v_mw_m$, $e_m=1$.
We thus have a $\flag \Sigma_a$-diagram of rings in the sense of Sections
\ref{sec:equivdiagrams}  and \ref{sec:equivringdiagrams}. 

\item Given a module $M$ over $\RRa$ we may define a module $w_*M$
 over $w_*\RRa$  by 
$$(w_*M) (K_0\supset \cdots \supset K_s)= M(K_0^w\supset \cdots \supset K_s^w).$$ 
An equivariang module is given by module maps $w_m: M \lra w_*M$ over the ring map $w_m: \RRa
 \lra w_*\RRa$ with  $(vw)_m=v_mw_m$, $e_m=1$.
\end{itemize}

\subsection{Equivariant diagrams of rings}
The previous section shows that  $\RRa$ is a $\mW G$-equivariant
$\flag \Sigma_a(\mT)$-diagram of rings so we may apply the apparatus of
Sections \ref{sec:equivdiagrams} and \ref{sec:equivringdiagrams}. 
 This means that $(\mW G)_K$ acts on $\RRa (K)$ by ring homomorphisms,
and we may form the twisted group ring $\RRa (K)[(\mW G)_K]$, or take
invariants under $(\mW G)_{K}^{e}=\mW W_G^eK$ and then let $(\mW G)_K/(\mW G)_{K}^{e}=W_G^dK$
act by ring homorphisms and form 
$$\RRtw (K)=\RRa (K)^{\mW W_G^eK}[W_G^dK] =H^*(BW_G^eK)[W_G^dK].$$

We observe that $\RRa$ extends to a  $\flag (\Sigma_a(\mT))$-diagram of rings via
$$\RRa (K_0\supset \cdots \supset K_s)=\cEi_{K_0/K_s}H^*(B\mT /K_s) .$$

We have commented that the Lie component structure  on  $\flag
(\Sigma_a)$ is normal. This allows us to extend $\RR_{inv}$ to a $\flag
(\Sigma_a(\mT))$-diagram with
$$\RRtw (F)=\RRa (F)^{\mW W_G^eF}[W_G^dF] .$$

It is worth making the values explicit.
\begin{lemma}
\label{lem:locinv}
For any flag $F=(K_0\supset \cdots \supset K_s)$ we have
$$\RRtw (F) =\cEi_{K_0/K_s}H^*(BW_G^eF)[W_G^dF]$$
where 
$$\cE_{K_0/K_s}=\{ e(V) \st V \in \mathrm{Rep}(W_G^eF), V^{K_0}=0\}.$$
\end{lemma}

\begin{proof}
Suppose $W$ is a finite group acting on a ring $R$
and   $S$ is a mulitiplicatively closed set closed under the action of
$W$. First, we note that inverting $S$ has the same effect as
inverting the elements $Ns=\prod_{w\in W}ws$, so that
$S^{-1}M=(NS)^{-1}M$.   Now observe
$$(S^{-1}M)^W=((NS)^{-1}M)^W=(NS)^{-1}(M^W),  $$
where the second equality uses the fact we are in characteristic
zero so that we may decompose $M$ into isotypical pieces and these
will not interact. 

Taking $W=\mW W_G^eF$, $R=H^*(B\mT /K_s)$ and $S=\cE_{K_0/K_s, \mT}$
this shows
$$\RRinv (F) =(N\cE_{K_0/K_s})^{-1}H^*(BW_G^eF). $$

Finally, we consider the effecting of inverting Euler classes. If we
choose  a representation $V$ of $W_G^e(F)$, its weights fall into
$W$-orbits. If the decomposition into weights is $V|_{\mT}=\bigoplus_i \alpha_i$
we have 
$$e(V)=e(\bigoplus_i \alpha_i)=\prod_ie(\alpha_i). $$
This is the product of orbit-products. Thus inverting $N\cE_{K_0/K_s, \mT /K_s}$ is equivalent to
inverting $\cE_{K_0/K_s, W_G^eF}$. 
\end{proof}

\begin{warning}
\label{warn:invnotloc}
In the case of the torus,  the ring $\RRa (K\supseteq L)$  is obtained
from $\RRa (L)$ by localization (i.e., by inverting $\cE_{K/L}$). This
is not the case for $\RRinv$. This is apparent even in the simple
example of Subsection \ref{subsec:so3}.
\end{warning}

\subsection{The rotation group}
\label{subsec:so3}
It is instructive to consider the example $G=SO(3)$, with $\mN =O(2)$,
$\mT =SO(2)$. We will display various data associated to a length 1
flag in rows. The first row is  a module $N$ over $\RRa$, the next
pair of rows gives $\RRa$, followed by the component structure
$W_{\bullet}^e$, a pair of rows for $\RRinv$ and finally an
$\RRinv$-module $M$.

We illustrate the structure for the particular flag $F=(\mT \supset
1)$ and its two length 0 subflags. Take particular note of 
  the fourth row, where we record the Weyl groups of $W_G^eK$ and
  $W_G^e(\mT \subseteq K)$, using the
abbreviation $W=\mW G$ (a reflection group of order 2). We will use 
this example in describing the functors relating $\cA
(G,toral)$ and $\cA (\mN, toral)$, so that $N$ is an equivariant
$\RRa$-module, potentially in $\cA (\mN, toral)$ and $M$ is an
$\RRinv$-module, potentially in $\cA (G, toral)$. As elsewhere
$H^*(BSO(2))=\Q [c]$ for an element $c $ of codegree 2 and 
$H^*(BSO(3))=\Q [d]$ for an element $d=c^2 $ of codegree 4. 

$$\diagram
N&& N(1) \rto & N(\mT \supset 1) & N(\mT) \lto \\
&&\Q [c] \rto \dto^= & \Q [c,c^{-1}]\dto^= & \Q\lto \dto^=  \\
\RRa &&\RRa(1) \rto & \RRa(\mT \supset 1) & \RRa(\mT) \lto \\
W_{\bullet}^e&&W  &  1& 1\\
\RRinv &&\RRinv (1) \rto & \RRinv (\mT \supset 1) & \RRinv (\mT) \lto \\
&&\Q [d] \rto \uto_= & \Q [c,c^{-1}]\uto_= & \Q\lto \uto_=  \\
M&&M(1) \rto & M(\mT \supset 1) & M(\mT) \lto 
\enddiagram$$
Note in particular that 
$$\RRinv (\mT \supset 1)=\Q [c,c^{-1}] \neq
\Q[d,d^{-1}]=\cEi_{\mT}\RRinv (1). $$

We should also consider the flag $\mT
\subset C_r$ for $r\geq 2$ so as to note the differences entailed by
the fact that $W(C_r)=\mW G$ is discrete and hence has trivial
identity component.  
$$\diagram
N&& N(C_r) \rto & N(\mT \supset C_r) & N(\mT) \lto \\
&&\Q [c] \rto \dto^= & \Q [c,c^{-1}]\dto^= & \Q\lto \dto^=  \\
\RRa &&\RRa(C_r) \rto & \RRa(\mT \supset C_r) & \RRa(\mT) \lto \\
W_{\bullet}^e&&1  &  1& 1\\
\RRinv&& \RRinv (C_r) \rto & \RRinv (\mT \supset C_r) & \RRinv (\mT) \lto \\
&& \Q [c] \rto \uto_= & \Q [c,c^{-1}]\uto_= & \Q\lto \uto_=  \\
M&& M(C_r) \rto & M(\mT \supset C_r) & M(\mT) \lto 
\enddiagram$$

\subsection{Subcategory conditions for $G$}
\label{subsec:qceG}

 The algebraic model $\cA (\mT)$ is the category of modules over the
diagram $\RRa$ of rings which  are subject to three conditions: (i)
quasi-coherence, (ii) extendedness and (iii) $\cF$-continuity. 

In view of Warning \ref{warn:invnotloc}, we must be explicit in
formulating the quasi-coherence and extendedness conditions for
equivariant $\RRinv$-modules on the poset $\flag (\Sigma_a)$.
We also observe that the conditions are compatible  with the 
$\mW G$-action.

Suppose then that we have flags
$$E=(K_0\supset K_1\supset \cdots \supset K_s)$$
and
$$F=(L_0\supset L_1\supset \cdots \supset L_t)$$
with $E\supset F$. This gives a ring map 
$$\RRinv (F)\lra \RRinv (E)$$ 
and for any $\RRinv$-module $M$ we have a structure map $M(F)\lra M(E)$.

In order to discuss quasi-coherence and extendedness, we introduce 
further terminology. This will be shown to be redundant, and not be
used after this subsection. 

\begin{defn}
(i) An $\RRinv$-module $M$  {\em follows the coefficients} if for any
pair of flags $E\supset F$ the structure map induces an isomorphism
$$\RRinv (E)\tensor_{\RRinv (F)}M(F)\cong M(E).$$
(ii) An $\RRinv$-module $M$ is {\em quasi-coherent} if it follows the
coefficients whenever $F=(K_s)$ is the singleton flag of the smallest
term in $E$
$$\RRinv (E)\tensor_{\RRinv (K_s)} M(K_s)\cong M(E). $$
(iii) An $\RRinv$-module $M$ is {\em extended} if it follows the
coefficients whenever $F=(K_0)$ is the singleton flag of the largest
term in $E$
$$\RRinv (E)\tensor_{\RRinv (K_0)} M(K_0)\cong M(E). $$
\end{defn}

\begin{remark}
(i) If $M$ is qce then it follows the coefficients for any inclusion
$E\supseteq F$ of flags. 

(ii) If $M$ follows the coefficients
for the addition of any single term to a flag then it is qce and
follows the coefficients in general. 

(iii) However if $M$ is qce for pairs this is not sufficient on its
own. For example we may consider the inclusion of a length 1 flag in a
length 2 flag: $(H\supset L) \lra (H\supset K \supset L)$. In this
case,  $\mW_{(H\supset K \supset L)} $ is typically a proper subgroup of
$\mW_{(H\supset L)}$ and so in general we have a proper containment
$$\RRa (H\supset K \supset L)^{\mW_{(H \supset L)}} =\cEi_{H/L}\RRa (L)^{\mW_{(H \supset L)}} 
\subseteq \cEi_{H/L}\RRa (L)^{\mW_{(H \supset K\supset L)}} = 
\RRa (H\supset K \supset L)^{\mW_{(H\supset K \supset L)}} .$$
The condition 
$$M(H\supset K \supset L)=
\RRa (H\supset K \supset L)^{\mW_{(H\supset K \supset
    L)}}\tensor_{\RRa (H\supset K \supset L)^{\mW_{(H \supset L)}}}
M(H\supset L)$$
is  a new condition, one not seen in the inclusion of a length 0 flag in a
length 1 flag.  
\end{remark}

The idea of $\cF$-continuity, is that it provides a uniform 
bound on denominators. In the original setting, the 
definition is that $\cF$-continuity requires a specified factorization for
each subgroup $K$:
$$\diagram 
&\cEi_K\prod_{L\subseteq K}M(L) \dto\\
M(K)\ar@{-->}[ur] \rto &\prod_{L\subseteq K}\cEi_{K/L}M(L) 
\enddiagram$$
and these should be compatible with composition. 
We note that the collection of subgroups involved in this condition depends $\mW
G$-equivariantly on $K$, and if the condition holds for $K$ it holds for any subgroup in
the  $\mW G$-orbit of $K$. 

We may now formulate the condition for $\RRinv$-modules. The equivariance will ensure that maps have
image in modules of invariants, so we avoid the use of invariants in
the statement.

\begin{defn}
 An $\RRinv$-module $M$ is {\em $\cF$-continuous} if there is a specified factorization for
each subgroup $K$
$$\diagram 
&\cEi_K\prod_{L\subseteq K}\RRa (L)\tensor_{\RRinv (L)}M(L)\dto\\
M(K)\ar@{-->}[ur] \rto &\prod_{L\subseteq K}\cEi_{K/L}\RRa (L)
\tensor_{\RRinv (L)}M(L) 
\enddiagram$$
and these should be compatible with composition. 
\end{defn}

\subsection{The model}
\label{subsec:model}
We are now ready to define the algebraic model $\cA (G, toral)$.
Throughout this subsection we use the diagram 
$\Sigma =\flag (\Sigma_a(\mT))$ and $\RRa$ is viewed as a
$\Sigma$-diagram of rings.

\begin{defn}
(i) The category of $\RRa[\mW G]$-modules is the category of
$\mW G$-equivariant  $\RRa$-modules.
In view of Lemma \ref{lem:WequivRmodisdiagram} we will not distinguish between
the model in which these are $\Sigma$-diagrams with the additional
structure of a $\mW G$-action and the model in which they are
 $\Sigma \sdr \mW G$-diagrams.  

(ii) The category $\cA (\mT )[\mW G]$ is the category
of $qce$, $\cF$-continuous $\mW G$-equivariant $\RRa$-modules.
\end{defn}

Now consider the Lie group component  structure $(\mW G)_{\bullet}^{e}$
on  $\mW G\sdr \Sigma$ and the quotient $(\mW G)_{\bullet}^{d}$
(see  Subsection \ref{subsec:Liecomponent}). This gives two diagrams
of rings.  Firstly, we have   the  $(\mW G)_{\bullet}^{d}$-diagram of invariants,
$\RRinv := \RRa^{(\mW G)_{\bullet}^{e}}$,   so that 
$$\RRinv (K) =H^*(BW_G^eK).$$  
Secondly, we have the $\Sigma$-diagram  $\RRtw$ of twisted group
rings,   whose value at a subgroup $K$ is 
$$\RRtw (K) =H^*(BW_G^eK)[W_G^dK].$$

\begin{defn}
\label{defn:AGtoral}
(i) The category $\cA_{inv} (G , toral)$ is the category
of $qce$, $\cF$-continuous $\RRinv$-modules.

(ii) The category $\cA_{tw} (G , toral)$ is the category of $qce$,
$\cF$-continuous modules over the diagram $\RRtw$ of rings.
\end{defn}

\begin{remark}
By Lemma \ref{lem:Wddiagramistw}, 
$\cA_{inv}(G,toral)\simeq \cA_{tw}(G,toral)$,
and as a matter of style we view $\cA_{inv}(G,toral)$ as the primary
one, abbreviating it to $\cA (G, toral)$.
\end{remark}

There is one special case where it is easy to describe the model of
toral spectra. 

\begin{lemma}
\label{lem:ANtoral}
The model for toral spectra simplifies when the identity component is
a torus to give 
$$\cA (\mN, toral)=\cA (\mT)[\mW G]$$
\end{lemma}

\begin{proof}
We need only observe that if $K\subseteq \mT$ then the identity
component of  $W_{\mN}K$ is a torus, and so it has trivial Weyl
group. The component structure is
therefore the trivial one, and $\RRinv=\RRa$. 
\end{proof}

\section{Toral $G$-spectra and toral $\mN$-spectra}
\label{sec:GspNsp}

We consider the algebraic counterpart of restriction from $G$-spectra
to $\mN$-spectra, and its right adjoint. 
We know from \cite{tnqcore} that the category of
$\mT$-spectra  is modelled by $\cA (\mT)$. It is rather clear (and
made explicit in  Lemma \ref{lem:piAisequiv})
that the module $M=\piA_*(X)$ in $\cA (\mT)$  arising from a
$G$-spectrum $X$ is  a $\mW G$-equivariant module. 

Our model $\cA (G,toral)$ has the property that the restriction from $G$-spectra
to $\mN$-spectra  is modelled by the functor $\theta_*$ defined in
Subsection \ref{subsec:Psitheta}. The purpose
of this section is to establish that the descent
adjunction (Proposition \ref{prop:Psitheta})  relating  $\RRinv$-modules and equivariant
 $\RRa$-modules continues to hold for the subcategories of  qce, $\cF$-continuous
modules. 

\subsection{From $\protect \cA (G,toral)$ to $\cA (\mT)[\mW G]$}

First we consider the algebraic counterpart of restriction. 

\begin{prop}
\label{prop:theta}
The functor
$$\theta_* :\mbox{$\RRinv$-modules} \lra \mbox{$\mW G$-equivariant-$\RRa$-modules} $$
preserves quasi-coherence,
extendeness and $\cF$-continuity and hence induces a functor
$$\theta_*: \cA (G,toral)\lra \cA (\mT)[\mW G].$$
\end{prop}

\begin{proof}
Suppose $M$ is an $\RRinv$-module with image $\theta_*M$ defined on a
flag $F$ by 
$$(\theta_*M)(F)=\RRa (F)\tensor_{\RRinv (F)} M(F). $$
We note that $\RRinv (F)=\RRa (F)^{(\mW G)_F^e}$, and $\RRa (F)$ is free
over $\RRinv (F)$. 
As in Lemma \ref{lem:locinv}, we note that a multiplicatively closed set $S$ preserved by
the action of a finite group has a cofinal multiplicatively closed
subset $NS$ whose elements are the products $Ns$ over orbits. Thus 
we will assume that the multiplicatively closed subsets are
invariant. Since $\theta_*M$ lies over $\mN$ the component strcuture is
trivial so the $\mW G$ action is entirely through equivariance (no
invariants are involved). Accordingly, it suffices to verify quasicoherence and
extendedness for pairs rather than more general flags. We will
write the proof in those terms since the subgroups concerned appear
 more directly.

If $M$ is quasi-coherent then the condition on cotoral pairs is that
the natural map induces an isomorphism $\RRinv (K\supseteq
L)\tensor_{\RRinv (L)}M(L)=M(K\supseteq L)$. It
follows that 
$$\begin{array}{rcl}
(\theta_*M)(K\supseteq L)
&=&\RRa (K\supseteq L)\tensor_{\RRinv(K \supseteq L)}
M(K\supseteq L) \\
&=&\RRa (K\supseteq L)\tensor_{\RRinv(K \supseteq L)}
\RRinv(K \supseteq L)\tensor_{\RRinv (L)}M(L) \\
&=&\RRa (K\supseteq L)\tensor_{\RRinv (L)}M(L) \\
&=&\RRa (K\supseteq L)\tensor_{\RRa (L)}\RRa (L)\tensor_{\RRinv
  (L)}M(L) \\
&=&\RRa (K\supseteq L)\tensor_{\RRa (L)}(\theta_* M)(L) 
\end{array}$$
and $\theta_* M$ is also quasi-coherent. 

If $M$ is extended then 
$$\RRinv (K\supseteq L) \tensor_{\RRinv (K)}M(K)=M(K\supseteq L).$$
For $\theta_*M$ we may then calculate
$$\begin{array}{rcl}
\RRa (K\supseteq L) \tensor_{\RRa (K)} (\theta_* M)(K)&=&\RRa (K\supseteq L) \tensor_{\RRa (K)}\RRa (K)\tensor_{\RRinv
(K)}M(K)\\
&=&
\RRa (K\supseteq L) \tensor_{\RRinv
(K)}M(K)\\
&=&
\RRa (K\supseteq L) \tensor_{\RRinv (K\supseteq L)}(\RRinv (K\supseteq
L )\tensor_{\RRinv
(K)}M(K))\\
&=&\RRa (K\supseteq L) \tensor_{\RRinv (K\supseteq L)}M(K\supseteq L)\\
&=&(\theta_* M) (K\supseteq L).
\end{array}$$
Thus $\theta_* M$ is also extended.

Supposing that $M$ is  $\cF$-continuous. Since  $\RRa
(F)$ is free of finite rank over $\RRinv (F)$, we may form the diagram 
$$\hspace*{-6ex}\diagram 
&\RRa (K)\tensor_{\RRinv (K)} \cEi_K\prod_{L\subseteq  K}M(L) \dto\rto &
\cEi_K\prod_{L\subseteq K}\RRa ( L)\tensor_{\RRinv (L)} M(L) \dto\\
\RRa (K)\tensor_{\RRinv (K)}M(K)\ar@[..][ur] \rto &\RRa (K)\tensor_{\RRinv (K)}\prod_{L\subseteq
K}\cEi_{K/L}M(L)\rto &
\prod_{L\subseteq K}\cEi_{K/L} \RRa (L)\tensor_{\RRinv ( L)} M(L). 
\enddiagram$$
The two right hand horizontals are induced by the $\RRinv (L)$-maps
$$M(L)\lra\RRa(L)\tensor_{\RRinv (L)}M(L)$$
using the universal property of  products.
The diagram shows that $\theta_*M$ is also $\cF$-continuous. 
\end{proof}

We will show in Proposition \ref{prop:thetaPsionG} that $\theta_*$  fits into a diagram 

$$\diagram
\mbox{toral-$G$-spectra} \rto^{\res^G_{\mN}} \dto_{\piAG_*} & \mbox{$\mN$-spectra} \dto^{\piAN_*}\\
\cA (G,toral) \rto^{\theta_*} \ar@{=}[d] &\cA (\mT)[\mW G] \ar@{=}[d]\\
\mbox{$\cF$-cts-qce-$\RRinv$-modules} \rto &\mbox{$\mW G$-equivariant-$\cF$-cts-qce-$\RRa$-modules}. 
\enddiagram$$

We conjecture that these maps to abelian categories can be
upgraded to Quillen equivalences with the associated differential
graded objects.


\subsection{Normal modules}

By contrast with $\theta_*$, the fact that the functor $\Psi$ takes qce modules to qce modules is rather subtle. 
Consider for instance the quasi-coherence associated to a cotoral
inclusion $K\supseteq L$. If the $\mW G$-equivariant $\RRa$-module $N$ 
is quasi-coherent, then $N(K\supseteq L)=\cEi_{K/L}N(L)$. We may
take $(\mW G)_{K\supseteq L}$-invariants of both sides, but since
$(\mW G)_{K\supseteq L}$ may be a proper subgroup of $(\mW G)_{L}$
this is not the quasicoherence condition for $\Psi N$, which states
instead  that
$$N(K\supseteq L)^{(\mW G)_{K\supseteq L}}=\cEi_{K/L}\RRa
(L)^{(\mW G)_{K\supseteq L}}\tensor_{\RRa(L)^{(\mW G)_L}}N(L)^{(\mW G)_L}.$$
In effect we need to be able to reconstruct modules from their
invariants using the ring $\RRa$. This is a special property not
enjoyed by all modules. 

We
suppose then that $W$ is a finite group acting on a $\Q$-algebra $R$. 

\begin{defn}
We say that a $W$-equivariant $R$-module $M$ is {\em normal} if the
natural map 
$$\nu: R\tensor_{R^W} M^W\lra M$$
is an isomorphism. 
\end{defn}

It is worth noting that normality is a strong condition. 
\begin{example}
(i)  Clearly if  $R=\Q G$ and $M$ is a non-trivial simple module then
$M$ is not normal. 

(ii) This also happens for modules that arise in our setting.  For instance
we may take $R=H^*(BSO(2))=\Q [c]$ with $W$ of order 2 acting to negate $c$,
so that $R^W=H^*(BSO(3))=\Q [d]$ with $d=c^2$. However it is easy to
see that the ideal $M=(c)$ is not normal (for example because the
inclusion $(c^2)\subseteq (c)$ is an isomorphism on $W$-fixed points
and the free module $(c^2)$ is normal). The fact that will give the
conclusion we need is that if $d$ is inverted everywhere (so $R=\Q
[c,c^{-1}]=M$) then we do obtain a normal module. 
\end{example}

There is an easy positive result.
 
\begin{lemma}
\label{lem:RRG}
If $R$ is free over $R^W$ then 
the class of normal $R$-modules is closed under the following
operations
\begin{itemize}
\item Arbitrary sums. 
\item Passage to kernels
\item Passage to cokernels
\item Passage to extensions
\end{itemize}

Any extended module of the form $M=R\tensor_S M'$ with $W$ acting
trivially on $S$ and $M'$ is normal. 
\end{lemma}

\begin{proof}
For example if $F_0$, $F_1$ are normal and $M$ is the cokernel of a
map $F_1\lra F_0$ we may form the diagram
$$\diagram 
F_1\rto &F_0\rto &M\rto &0\\
R\tensor_{R^G}F_1^W\uto^{\cong}\rto
&R\tensor_{R^G}F_0^W\uto^{\cong}\rto &R\tensor_{R^W}M^W\uto \rto &0
\enddiagram$$
Because we are in characteristic 0, passage to $W$ fixed points is
exact, and by hypothesis $R$ is flat over $R^W$, so the isomorphism
follows from the short 5-lemma. 

The other cases are similar. For an extended module of the given form 
$M^W=(R\tensor_SM')^W=R^W\tensor_S M'$ and normality is clear. 
\end{proof}

We will show that the modules that occur in an object $N$ of $\cA (\mN ,
toral)$ are close enough to being normal to ensure that $\Psi N$ is
qce. The following examples show that this is somewhat less
restrictive than might be expected. 

\begin{example}
(i) We have seen that $H^*(BSO(2)) =\Q [c]$ is a free module over
$H^*(BSO(3))=\Q[d]$. More precisely 
$$\Q [c]=\Q [d]\tensor (\eps \oplus \Sigma^2 \sigma)$$
where $\eps$ is the trivial module  and 
$\sigma$ is the sign module. 
If we ignore grading then $\Q [c]=\Q [d][W]$. 

In any case it follows by decomposing $V$ into $W$-isotypical pieces that any $\Q [c]$-module of the form
$\Q[c,c^{-1}]\tensor V$ is normal. 

The relevance of this is that it shows the model of $\{ 1,
\mT\}$-$SO(3)$-spectra (i.e., of spectra with geometric isotropy in
$\{ 1, \mT\}$) behaves well. Indeed,  we may consider an object 
$$X=(N\stackrel{\beta}\lra \Q [c,c^{-1}]\tensor V)$$
 of the model
of $\{ 1, \mT\}$-$\mN$-spectra; this means $N$ is a $\Q [c][W]$-module
and $V$ is a $\Q [W]$-module with the map $\beta$ being inversion of $c$. By
the above argument,  $N$ is normal, and it follows that 
 $$\Psi X=(N^W\lra \Q [c,c^{-1}]\tensor V) $$
is qce.

(ii) Similarly for the rank 2 group $SU(3) $ with maximal torus
$ST(3)$ and Weyl group $\mW G=\Sigma_3$,  where 
$$H^*(BST(3))=H^*(BSU(3))\tensor (\eps \oplus \Sigma^2 \mu \oplus
\Sigma^4 \mu \oplus \Sigma^6\sigma ) ,  $$
where $\sigma$ is the nontrivial simple repesentation of degree 1 and
$\mu $ that of degree 2. If we ignore the grading then 
$$H^*(BST(3))=H^*(BSU(3))[\mW G] .$$
One may check that if we invert linear forms then any module of the
form $H^*(BST(3))\tensor V$ is normal (the case $V=\mu$ is most interesting).
\end{example}

It seems natural to expect that with linear forms inverted, the module $H^*(B\mT)\tensor V$
is normal for any compact Lie group $G$, and it may be that more
general statements could be formulated giving the result that $\Psi N$
is qce directly as was done for $G=SO(3)$ in the above
example. However, we will instead  use injective resolutions to reduce the
verification to special cases. 

\subsection{From $\protect \cA (\mT )[\mW G]$ to $\cA (G,toral)$}
After our discussion of normal modules we are equipped  to turn to the
right adjoint $\Psi$. 

\begin{prop}
\label{prop:Psi}
The functor 
$$\Psi :\mbox{$\mW G$-equivariant-$\RRa$-modules} \lra \mbox{$\RRinv$-modules} $$
 takes quasi-coherent, extended modules to quasi-coherent
extended modules and  preserves $\cF$-continuous modules and hence induces a functor
$$\Psi: \cA (\mT)[\mW G]\lra \cA (G,toral).$$
\end{prop}

\begin{remark}
The functor $\Psi$ does not preserve quasi-coherence or extendedness separately.
\end{remark}

\begin{proof}
First, $\cF$-continuity is straightforward, since $M(K)^{(\mW G)_K}$
maps into the $(\mW G)_K$-invariants inside the $(\mW G)_{K \supseteq
L}$-invariants, and we have already observed that the passage to
invariants commutes with products and localizations.  The main issue
is the qce property, which is rather delicate.

Suppose $N$ is an $\RRa$-module with image $\Psi N$ defined by 
$$(\Psi N)(F)=N(F)^{(\mW G)_F^e}.$$
As in Lemma \ref{lem:locinv},  we note that a multiplicatively closed set $S$ preserved by
the action of a finite group has a cofinal multiplicatively closed
subset $NS$ whose elements are the products $Ns$ over orbits. Thus 
we will assume that the multiplicatively closed subsets are
invariant. 

Now suppose $E\supseteq F$. Since $N$ is qce we have
$$N(E)=\RRa (E)\tensor_{\RRa (F)}N(F). $$
Taking fixed points under $(\mW G)_E^e$ we have 
$$(\Psi N)(E)=N(E)^{(\mW G)_E^e}=
\left[ \RRa (E)\tensor_{\RRa
   (F)}N(F)\right]^{(\mW G)_E^e}. $$
Since the connected structure is decreasing $(\mW G)_E^e\subseteq (\mW
G)_F^e$ and  we need to show the natural map 
\begin{multline*}
\nu_{E\supset F}: \RRinv (E) \tensor_{\RRinv (F)}(\Psi N)(F)=
\RRa (E)^{(\mW G)_E^e}\tensor_{\RRa
(F)^{(\mW G)_F}}N(F)^{(\mW G)_F^e}\lra \\ 
\left[ \RRa (E)\tensor_{\RRa
(F)}N(F)\right]^{(\mW G)_E^e} =(\Psi N)(E). 
\end{multline*}
is an isomorphism. 

The character of the problem is like that of normality, and we adopt a
similar strategy.  We first note that the question of whether
$\nu_{E\supset F}$ is an isomorphism only depends on $N$ only through
$N(F)$, which is a $(\mW G)_F$-equivariant $\RRa (F)$-module.

\begin{lemma}
The class of modules $N(F)$ for which $\nu$ is an isomorphism is  closed under the following
operations
\begin{itemize}
\item Arbitrary sums. 
\item Passage to kernels
\item Passage to cokernels
\item Passage to extensions
\end{itemize}
It is an isomorphism for $N(F)=\RRa (F)$.
\end{lemma}

\begin{proof} This  clear for sums, and it is clear for $N(F)=\RRa
  (F)$. We illustrate the other cases by the passage to
  kernels. Suppose then that $\nu$ is an isomorphism for $N(F)=B,C$
  and that we have an exact sequence  
$$0\lra A\lra B \lra C$$
Taking $\mW_F=(\mW G)_F$-invariants, and tensoring with $\RRa
(E)^{\mW_E}$ over $\RRa (F)^{\mW_F}$ we obtain the first row in
the following diagram, and similarly the second row is obtained by
tensoring with $\RRa (E)$ and taking $\mW_E$-invariants. The
result follows from the short 5-lemma.
$$\diagram 
0\rto &
\RRa (E)^{\mW_E}\tensor_{\RRa (F)^{\mW_F}}A^{\mW_F}\rto \dto  
&\RRa (E)^{\mW_E}\tensor_{\RRa (F)^{\mW_F}}B^{\mW_F}\rto \dto  &\RRa (E)^{\mW_E}\tensor_{\RRa (F)^{\mW_F}}C^{\mW_F} \dto  \\
0\rto &
\left[ \RRa (E)\tensor_{\RRa (F)}A\right]^{\mW_E} \rto &\left[
  \RRa (E)\tensor_{\RRa (F)}B\right]^{\mW_E} \rto &
\left[ \RRa (E)\tensor_{\RRa (F)}C\right]^{\mW_E} 
\enddiagram$$
\end{proof}

In effect the lemma says that the result is only obvious when $N(F)$
is a free $\RRa (F)$-module. The strategy of proof is to reduce to the
case of certain standard injectives that we identify
precisely. We  note that these standard injectives come from the polynomial rings $H^*(B\mT
/K)$. Because the polynomial ring  $H^*(B\mT /K)$ is Gorenstein, the
injective is also the local cohomology and we can deduce this case
from  that of the free module. 

In more detail, we show in Section 
\ref{sec:AGtoralinj} that any module $N$  admits an injective 
presentation $0\lra N \lra I_0 \lra I_1$ where $I_0$ and $I_1$ are sums
of $\mW G$-equivariant injectives of a particular form.  It therefore 
suffices to prove the result for  the special case of these basic
injectives. These are discussed in detail in Section
\ref{sec:AGtoralinj}, but we will summarize the properties we need
here to avoid interrupting the thread of the argument.

Suppose  then that $K\subseteq \mT$ and consider  a basic injective
arising from $K$. This is obtained from an injective module $I$ over
$H^*(B\mT /K)$, namely $$I=H_*(B\mT /K^{L\mT/K}).$$ 
 Indeed, the right adjoint $f_{K}^{\mT}$ to evaluation at $K$ gives a
an injective  $f_{K }^{\mT}(I)$ in $\cA (\mT)$ and then we may
coinduce the module to $\mN$, where it takes the form
$$f_{(K) }^{\mN}(\Q [W]\tensor I)=f_{(K) }^{\mN}(I)\tensor \Q [W]$$
 in $\cA (\mN, toral)=\cA (\mT )[\mW G]$. Notice that the value of
 this injective at any flag is free over $\Q [W]$. 

Of course $N (H) =0$ unless $H $ is subconjugate to $K$. From 
the qce condition it follows  that the value $N(F)$ is zero unless $K\supseteq
L_0$.  We note that if $K \not \subseteq K_0$ then $I$ is
$\cE_{K_0/L_0}$ torsion; as observed elsewhere, we can always localize
with respect to products over $W$-orbits.  Thus the qce  condition for $\Psi N$
holds for such flags. We may therefore suppose that $K\supseteq K_0$.

We may be explicit about the value. Indeed,  $K\supseteq L_0\supseteq \cdots \supseteq L_t$ and 
$\mT/L_t=\mT /K \times K/L_t$. Thus 
\begin{multline*}
N(F) =\cEi_{L_0/L_s}H^*(B\mT/L_s)\tensor_{H^*(B\mT /A)}H_*(B(\mT
/K)^{L (\mT /K)}) [\mW G]=\\
\cEi_{L_0/L_s}H^*(BK/L_s)\tensor_{\Q}H_*(B(\mT
/K)^{L(\mT /K)}) [\mW G]. 
\end{multline*}
It remains to observe that $\nu_{E\supset F}$ is an isomorphism for
this $N(F)$.  We will first verify the statement without $\mW G$. 
For this we apply the following lemma to $T=\mT /K$.

\begin{lemma}
Suppose  $\fW$ is any finite subgroup of $Aut (T)$ consider the category $\fW$-equivariant $H^*(BT)$-modules.

If $BT^{LT}$ is the Thom space of the tangent space $LT$ of $T$ at
$e$, the module $H_*(B T^{LT})$ is the cohomology of a finite complex of
$H^*(BT)[\fW]$-modules each term of which  generated by $H^*(BT)$ using direct sums,
cokernels and direct limits. 
\end{lemma}

\begin{remark}
Note that the insertion of the adjoint representation $LT$ is necessary. For example if $T=SO(2)$
is the circle and $W=\mW SO(3)$ is of order 2, $H_*(BT^{LT})$ is a
suspension of the dual of $(c)$, and we have the exact sequence
$$0\lra \Q [c]\lra \Q [c, c^{-1}]\lra (c)^{\vee}\lra 0$$
proving the lemma in this case. On the other hand  the module
$k[c]^{\vee}$ is not in this category since $\theta_*\Psi
(k[c]^{\vee})\not\cong k[c]^{\vee}$.
\end{remark}

\begin{proof}
To start with,  ignore the action of $\fW$.
If we choose a finite set of $\fG$  of generators of the ideal $\fm$ of elements of
$R=H^*(BT)$ of positive codegree, we may form the stable Koszul complex
$K^{\bullet}_{\infty}(\fG)$, with 
$$K^{n}_{\infty}(\fG)=\bigoplus_{\tau \subseteq \fG, |\tau|=n}
R[\frac{1}{\prod_{g \in \tau}g}]. $$
The point of the stable Koszul complex is that it calculates local
cohomology, so that if $T$ is of rank $s$, we have 
$$H^*(
K^{\bullet}_{\infty}(\fG))=H^*_{\fm}(R)=H^s_{\fm}(R)=H_*(BT^{LT}). $$

Now choose $\fG$ so that the construction is
$\fW$-equivariant. Indeed, adding translates as necessary, we choose $\fG$ to be a union of $\fW$-orbits, and group the terms
in $K^{n}_{\infty}(\fG)$ into $\fW$-orbits of $n$-tuples $\tau$. Thus
if the orbit $\cO$ of $\tau$ has isotropy $\fV$ we find
$$K^{\cO}_{\infty}=\bigoplus_{\tau\in \cO}R[\frac{1}{\prod_{i \in
    \tau}g_i}]=\fW \tensor_{\fV} R[\frac{1}{\prod_{g \in
    \tau}g}]$$
\end{proof}

Finally, we argue that we can insert the group ring $\Q [W]$. Indeed,
we are considering the map   
$$\nu_{E\supset F}: \RRa (E)^{(\mW G)_E^e}\tensor_{\RRa
(F)^{(\mW G)_F^e}}N(F)^{(\mW G)_F^e}\lra  
\left[ \RRa (E)\tensor_{\RRa
(F)}N(F)\right]^{(\mW G)_E^e}. $$
We have observed that if $\nu_{E\supset F}$ is an isomorphism for 
$N(F)=\RRa (F)$ then it is also an isomorphism when  $N(F)$
comes from $f_{(K)}^{\mN}(I)$ with $I=H_*(B\mT/K^{L\mT /K})$. We now
show that, similarly, if $\nu_{E\supset F}$ is an isomorphism for
$N(F)=\RRa (F)[W]$ then it is also an isomorphism when $N(F)$ comes 
from $f_{(K)}^{\mN}(I[W])$. For the case $N(F)=\RRa (F)[W]$ let
us  note that $N_G(E)\subseteq N_G(F)$; this gives a map of Weyl
groups $W_G(E)\lra W_G(F)$, and  passing to quotients under their
respective maximal tori, we have an inclusion $W_G(E)/(\mT /K_s)\subseteq
W_G(F)/(\mT /L_t)$ of coset spaces. 

Now for any connected Lie group $\Gamma$ with maximal torus $T$, the
rational Serre spectral sequence
of $\Gamma/T \lra BT \lra B\Gamma$ collapses to give an isomorphism 
$$H^*(BT)\cong H^*(B\Gamma) \tensor H^*(\Gamma /T) $$
of $H^*(B\Gamma )[W]$-modules.  Furthermore the Weyl group acts trivially on the first
factor. For example
$$H^*(B\mT/K_s)=H^*(BW_G^eE) \tensor H^*(N_G^e(E)/\mT), $$
so that when we  invert $\cE_{K_0/K_s}$ we find 
$$\RRa (E)= (\Psi \RRa) (E) \tensor H^*(N_G^e(E)/\mT)$$

Using this we may  identify $\nu_{E\supseteq F}$  as  
\begin{multline*}
\Psi \RRa (E)\tensor_{\Psi \RRa (F)} [\Psi \RRa
(F)\tensor H^*(N_G^e(F)/\mT) [W] ]^{(\mW G)_F^e}\stackrel{\nu_{E\supset F}}\lra  
\left[ \RRa (E)\tensor_{\RRa
(F)}\RRa (F)[W]\right]^{(\mW G)_E^e}\\
=
\left[ \RRa (E)[W]\right]^{(\mW G)_E^e}
=
(\Psi \RRa)(E)\tensor \left[ H^*(N_G^e(E)/\mT) [W]\right]^{(\mW
 G)_E^e}.
\end{multline*}
This compares two free $\Psi \RRa(E)$ modules obtained by tensoring
with the vector spaces
$$ [ H^*(N_G^e(F)/\mT) [W] ]^{(\mW G)_F^e} \mbox{ and }
\left[ H^*(N_G^e(E)/\mT) [W]\right]^{(\mW  G)_E^e}$$
We note that they are both vector spaces of dimension $|W|$ (they are
not isomorphic as {\em graded} vector spaces, but $E\neq F$ so $\RRa (E)$ is 2-periodic and
tensoring gives abstractly  isomorphic $\Psi \RRa (E)$-modules).

Finally, we observe that $\nu $ is obtained from a
$\mW G_E$-equivariant $\RRa (E)$-module map 
$$\nu_{E\supset F}: \RRa (E) \tensor_{\RRa
(F)^{(\mW G)_F^e}}\left[\RRa (F)[W]\right] ^{(\mW G)_F^e}\lra  \RRa (E)\tensor_{\RRa
(F)}\RRa (F)[W] $$
by passage to $\mW G_E^e$-fixed points. This map is surjective since 
$\RRa (E)[W]$ is generated as an  $\RRa (E)[\mW G_E]$-module by 
$\left( \RRa (F) [W]\right)^{\mW G_F^e}$.  Hence $\nu$ is an
isomorphism as required. 
\end{proof}

\subsection{Toral descent from $G$ to $\mN$}
The descent property now follows from the result for arbitrary
modules. 

\begin{cor}
\label{cor:Psitheta}
We have an  adjunction 
$$\adjunction
{\theta_*}{\cA (G,toral)}{\cA (\mT )[\mW G]}{\Psi},  $$
for which the unit is an isomorphism.
\end{cor}

\begin{proof}
In the light of Lemmas \ref{prop:theta} and \ref{prop:Psi}, this is
immediate from Proposition \ref{prop:Psitheta}
\end{proof}

\section{Homological algebra of $\protect \cA (G,toral)$}
\label{sec:AGtoralinj}

In this section we deduce the facts we need about the homological algebra from  $\cA (G, toral)$ from known
properties of $\cA (\mT)$. In particular,  we show it has finite injective
dimension equal to the rank. 


\subsection{Right adjoints to evaluation}
\label{subsec:fK}
  The study of $\cA (\mT)$ in  \cite{tnq1}  shows  that $\cA (\mT)$
  has sufficiently many injectives. Indeed,  it is shown that enough
  injectives can be imported  from module categories using right
  adjoints $f_K^{\mT}$  to evaluation at  subgroups $K$. We will not
  repeat the argument  here in detail, but the idea is to argue by
  induction on the {\em supporting codimension}:
$$\suppcod (M):=\mathrm{min}\{ \dim (\mT /K)\st M(K)\neq 0\}$$
of a nonzero modules $M$.  One may find a map from any module $M\neq 0$ to
  a sum of injectives $f_K^{\mT}(I)$ which is a monomorphism at
  subgroups of  codimension $\suppcod (M)$. The general case can be built up from
  this. Accordingly, it suffices here to  discuss the right adjoints to
  evaluation. 

  The starting point from \cite{tnq1} is that for any closed subgroup
  $K\subseteq \mT$ of codimension $c$, there is a right adjoint $f_K^{\mT}$ to
  evaluation at $K$:
$$\adjunction{eval_K}
{\cA (\mT)_{\suppcod \geq c}}
{\mbox{torsion-$H^*(B\mT/K)$-modules}}
{f_K^{\mT}}.$$

We may combine these phenomena over a $\mW G$-orbit $(K)$. The point
is that the distinct subgroups $K_i$ in the orbit are of the same
codimension in $\mT$ and hence only cotorally related if they are equal. 
$$\adjunction{eval_{(K)}}
{\cA (\mT)_{\suppcod \geq c}}
{\prod_{K'\in (K)}\mbox{torsion-$H^*(B\mT/K')$-modules}}
{f_{(K)}^{\mT}}.$$
This is compatible with the $\mW G$-action. To describe the structure, note that we have an inclusion 
$(K)\lra \Sigma_a(\mT)$ of posets with $\mW G$-action. Because  $(K)$ is
a discrete poset it is reasonable to write $H^*(B\mT/(K))$ for the
restriction of  $\RRa$ to $(K)$. Since $(K)$ is a transitive $\mW
G$-set, 
there is an equivalence
$$\mbox{$H^*(B\mT/(K))[\mW G]$-modules} \simeq  \mbox{$H^*(B\mT/K)[(\mW
G)_K]$-modules}. $$
Thus we have an adjunction
$$\adjunction{eval_{(K)}}
{\cA (\mT)[\mW G]_{\suppcod \geq c}}
{\mbox{torsion-$H^*(B\mT/(K))[\mW G]$-modules}}
{f_{(K)}^{\mN}}.$$
We will generally specify the particular subgroup $K$ and take the
argument of $f_{(K)}^{\mN}$ to be a $H^*(B\mT /K)[(\mW G)_K]$-module. The
right adjoint to evaluation on $\cA (G,toral)$ can now be defined
 in terms of the functor for $\mN$.

\begin{lemma}
The right adjoint to evaluation at $K$ is given 
 by the formula 
$$f^G_{(K)}(M)=\Psi f^{\mN}_{(K)}(\theta_*M),  $$
where $M$ is an $H^*(BW_G^eK)[W_G^dK]$-module. 
We have the commutative diagram
$$\diagram
\cA (\mT)[\mW G]_{\suppcod \geq c} \dto_{\Psi}&
\mbox{torsion-$H^*(B\mT/K)[(\mW G)_K]$-modules}\dto^{\Psi^{\mW W_G^eK}} \ar[l]_-{f^{\mN}_{(K)}} \\
\cA (G, toral)_{\suppcod \geq c} &\mbox{torsion-$H^*(BW_G^eK)[W_G^d(K)]$-modules} .\ar[l]^-{f^G_{(K)}}
\enddiagram$$
\end{lemma}

\begin{proof}
We make the calculation
$$\begin{array}{rcl}
\Hom_{\cA (G,toral)} (X, f^G_{(K)}(M)) &=&\Hom_{\cA (G,toral)} (X, \Psi f^{\mN}_{(K)}(\theta_*M))\\
&=&\Hom_{\cA (\mT)[\mW G]} (\theta_* X,  f^{\mN}_{(K)}(\theta_*M))\\
&=&\Hom_{H^*(B\mT /K)[\mW G_K]} ((\theta_* X)(K),  \theta_*M)\\
&=&\Hom_{H^*(B\mT /K)} (H^*(B\mT/K)\tensor_{H^*(BW_G^eK)}X(K),  \theta_*M)^{\mW G_K}\\
&=&\Hom_{H^*(BW_G^eK)} (X(K),  H^*(B\mT /K)\tensor_{H^*(BW_G^eK)}M)^{\mW G_K}\\
&=&\Hom_{H^*(BW_G^KK)} (X(K),  M)^{W_G^dK}
\end{array}$$

\end{proof}

\subsection{The category $\cA (\mN , toral)$}
The evaluation functors immediately bring $\cA (\mN , toral)$ under
control. 
 
\begin{lemma}
The abelian category $\cA (\mN ,toral)=\cA (\mT)[\mW ]$ has enough
injectives and is of  injective dimension equal to the rank.
\end{lemma}

\begin{proof}
In the category of  $H^*(B\mT/K)[(\mW
G)_K]$-modules, any torsion injective embeds in  
$$\Hom_{\Q}(\Q [(\mW G)_K], H_*(B\mT /K))=(H^*(B\mT /K)[(\mW
G)_K])^{\vee}. $$
Applying $f_{(K)}^{\mN}$ we obtain enough injectives in $\cA (\mT
)[\mW G]$. 

Since 
$$\Hom_{\cA (\mT)[\mW G]}(M,N)=\Hom_{\cA (\mT)}(M,N)^{\mW G}$$
and passage to fixed points is exact, it follows that the injective
dimension of $\cA (\mT)[\mW G]$ is no more than that of $\cA (\mT)$. 
The case of coinduced modules shows they are equal. 
\end{proof}

\subsection{The category  $\protect \cA (G, toral)$}
The properties we want for $\cA (G,toral)$ itself can now be deduced
formally from what we have proved for $\cA (\mN , toral)$.

\begin{prop}
\label{prop:AGtoralinj}
The abelian category $\cA (G,toral)$ has enough injectives and is of injective dimension equal to
the rank of $G$.  
\end{prop}

\begin{proof}
Since we are working over the rationals, $H^*(B\mT/K)$ is free over
$H^*(BW_G^eK)$ and $\theta_*$ is exact.  The right adjoint $\Psi$
therefore preserves injectives, and  $\Psi I$ is injective in $\cA (G,
toral)$ for every injective $I$ in $\cA (\mT )[\mW G]$. Consequently, 
if we apply $\Psi$ to an injective resolution of $M$ we obtain an
injective resolution of $\Psi M$. Since the unit of the adjunction is
an isomorphism (Proposition \ref{prop:Psitheta} and Corollary \ref{cor:Psitheta}), all objects of $\cA (G, toral)$ are in the image of
$\Psi$ and there are enough injectives in $\cA (G,toral)$. 

Since  $\cA (\mT)$ is of  injective dimension is $r$ \cite{tnq2}
it follows that  $\cA (G, toral)$ is of injective dimension
$\leq r$. To see that this bound is achieved, we may consider free
spectra (which is to say torsion modules over the polynomial ring
$H^*(BG_e)$ on $r$ generators), or more specifically $G_+$ (which is
to say the torsion module $\Q [G_d]$).
\end{proof}

\part{Topology}
\section{Toral detection}
\label{sec:toraldetect}

We show that the toral part of $G$-spectra is detected in
$\mT$-equivariant homotopy. This is the key result that makes this
entire approach viable.

\subsection{Idempotents}
Underlying the structure of any monoidal category is the endomorphism
ring of the unit object, which in  our case is the ring of stable homotopy
groups of $S^0$. Accordingly,   we recall how the Burnside ring $A(G)=[S^0,S^0]^G$ is related to
spaces of subgroups. Given a stable map $f: S^0\lra S^0$, the degree
of geometric fixed points defines a function $\deg (f): \cF (G) \lra
\Z$ from the set $\cF (G)$ of
subgroups of $G$ with finite index in their normalizers. It is clearly
constant on conjugacy classes, and one may show that $\deg (f)$ is
continuous in the Hausdorff metric topology. It was shown by tom Dieck
\cite{tD} that the
map
$$A(G)\lra C_G(\cF (G), \Z)$$
is injective and that it a rational isomorphism. Furthermore  
$C_G(\cF (G), \Z) \tensor \Q \cong C_G(\cF (G), \Q)$.
Finally, it is easy to deduce the degree  of the geometric fixed
points under any subgroup:  if $K$ is not of finite index in its normalizer then $\deg
(f^K)=\deg (f^H)$ whenever $K$ is cotoral in $H$.

Next we note that the conjugacy class of maximal tori is open and
closed in $\cF (G)$, so there is an idempotent $e_{\mT} \in A(G)$ with
support on $(\mT)$ and the degree of its $K$-fixed points is 1 for
subgroups of a maximal torus and 0 otherwise.

We may then localize with respect to $e_{\mT}S^0$, and obtain 
$$\mbox{toral-$G$-spectra}=e_{\mT}\left[ \mbox{$G$-spectra}\right].$$

\begin{lemma}
\label{lem:etoralp}
Writing $\Lambda (\mT)$ for the family of subgroups of some maximal
torus, the natural map $\etoralp \lra S^0$ induces an equivalence
$\etoralp \simeq
e_{\mT} S^0$.
\end{lemma}

\begin{proof}
By definition the $K$-fixed point space of $\etoralp$ is equivalent to $S^0$ if
$K$ lies in a maximal torus and is a point otherwise. The map is
therefore an equivalence in geometric $K$-fixed points for all $K$ and
hence a weak equivalence.  
\end{proof}

\begin{cor}
\label{cor:etoralp}
$$[\etoralp, \etoralp]^{G}=[S^0,S^0]^{\mT}=\Q$$
which is detected by the degree in homotopy of
geometric $\mT$-fixed points. 
\end{cor}

\begin{proof}
After Lemma \ref{lem:etoralp}, we see 
$$[\etoralp,
\etoralp]^G=[\etoralp, S^0]^G=[e_{\mT}S^0, S^0]^G=e_{\mT}A(G)=\Q. $$
\end{proof}

\subsection{Toral restriction is faithful}
The key to our strategy is that the restriction from $G$ to a
maximal torus $\mT$ is faithful on toral spectra. 

\begin{prop}
\label{prop:toraldetection}
The forgetful map 
$$[X,Y]^G\lra [X,Y]^{\mT}$$ 
is rationally split injective if $X$ is a $\toral$-spectrum. 
\end{prop}

\begin{proof} Under the natural equivalence $[G/\mT_+\sm X, Y]^G=[X,
  Y]^{\mT}$ the forgetful map corresponds to the the projection $\pi: G/\mT
  \lra *$. 

Since $X$ is a $\toral$-spectrum, it is equivalent to $X \sm
\etoralp$, so that a splitting can be obtained from a factorization 
$$\diagram \etoralp &G/\mT_+ \sm \etoralp \lto_(0.58){\pi}\\
&\etoralp \uto_s \ulto^1
\enddiagram$$
It remains to choose a suitable $s$, and we note that Corollary
\ref{cor:etoralp} shows we need only show $s$ is non-trivial in
$\mT$-geometric fixed points. In fact we will show that maps in this
pattern are determined by $\pi_0$ of $\mT$-geometric fixed points. 

\begin{lemma}
\label{lem:PhiTdetect}
For the three
pairs of  spaces $X,Y$ in the above diagram  we have an isomorphism 
$$[X,Y]^G\stackrel{\cong}\lra \Hom_{\Q W}(\pi_0(\Phi^TX) ,
\pi_0(\Phi^TY)).$$
\end{lemma}
\begin{proof}
This is already done for the edges labelled $1, \pi$, so we only need
to deal with the edge labelled $s$ where $X=\etoralp$, $Y=G/\mT_+ \sm \etoralp$.

Write $L$ for the representation of $\mT$ given by the tangent space
to $G/\mT$ at $e\mT$, and note the fact that $\mT$ is a maximal
abelian connected subgroup shows that $L^{\mT}=0$. 
The Wirthm\"uller adjunction gives isomorphisms
\begin{multline*}
[\etoralp, G/T_+\sm \etoralp]^G =
[\etoralp, F_T(G_+, S^L\sm \etoralp)]^G \\
=[\etoralp, S^L\sm \etoralp]^{\mT}
=[S^0, S^L]^{\mT} =\Q. 
\end{multline*}
The last isomorphism follows from the Segal-tom Dieck splitting,
since the only subgroup $K$  of $\mT$ with finite index in its
normalizer is $\mT$ itself. 
Following through the adjunctions, the composite isomorphism is given
by forgetting from $G$ to $\mT$ and  composing with the $\mT$-map
$$G/T_+\sm \etoralp \lra S^L\sm \etoralp ; $$
induced by the Pontrjagin-Thom map $G/T\lra S^L$.
It follows that maps are detected by degree in $\mT$-geometric fixed points.  
\end{proof}

According to Lemma \ref{lem:PhiTdetect},  we need only consider 
$$\diagram 
\Q &\Q \mW G \lto_{\pi_*}\\
&\Q \ulto^1 \uto_{s_*}
\enddiagram$$
and select  $s$ so that $|\mW G| s_*$ is the norm map.
\end{proof}

\begin{remark}
To be more specific, we can take $|\mW G|s$ to be the composite
$$\etoralp \lra F_{\mT}(G_+, \etoralp) \stackrel{F_{\mT}(G_+,
  i_L)}\lra F_{\mT}(G_+, S^L\sm \etoralp)\simeq G/\mT_+ \sm
\etoralp,  $$
where the first map is the adjunct of the identity and the last is the
standard Wirthm\"uller equivalence. 
\end{remark}

\section{Borel cohomology and the associated homology theory}
\label{sec:Borel}

\subsection{Classical isomorphisms}
The essential ingredients in the proof that $\cA (G,toral)$ provides
an effective invariant are classical facts about the
cohomology of the Borel construction. We will need to apply the results to
$W_G^eK$ for various subgroups $K$ of $G$, so in this section we take
$\Gamma$ to be a  compact Lie group with maximal torus $\mT$, 
$\mN =N_{\Gamma}\mT$ and Weyl group $\mW \Gamma =N_{\Gamma}\mT/\mT$.

\begin{lemma} 
\label{lem:class}
  If $Z$ is a free $\Gamma$-space then we have natural isomorphisms

(i) $H^*(Z/\mN )\cong H^*(Z/T)^{\mW \Gamma} $ 

(ii) $H^*(Z/\Gamma )\cong H^*(Z/\mN ) $

(iii) $H^*(Z/\Gamma)\cong H^*(Z/T)^{\mW \Gamma}$ and 

(iv) If $\Gamma$ is connected, there is a natural isomorphism
$$H^*(BT)\tensor_{H^*(B\Gamma)} H^*(Z/\Gamma)\stackrel{\cong}
\lra H^*(Z/T). $$
\end{lemma}

\begin{proof}
It suffices to treat the unbased case. 

Part (i) follows since the Serre spectral sequence $Z/T \lra Z/\mN
\lra B\mW \Gamma$ collapses when the group order is invertible. 

Part (ii) follows from the Serre spectral sequence of 
$\Gamma /\mN \lra X/\mN \lra X/\Gamma $, since $\Gamma /\mN$ is
rationally contractible. 

Part (iii) follows by combining Parts (i) and (ii). 

Part (iv) follows from the Eilenberg-Moore spectral sequence of the
pullback square 
$$\diagram 
Z/\mT \rto \dto & B\mT \dto \\
Z/\Gamma  \rto & B\Gamma
\enddiagram$$
We note that connectedness of $\Gamma$ ensures $B\Gamma$ is
1-connected, and working over $\Q$ ensures that $H^*(B\mT )$ is free
over $H^*(B\Gamma)$.
\end{proof}

\begin{cor}
\label{cor:NGamma}
For any $N$ spectrum $B$, the map $i: B\lra \Gamma_+ \sm_{\mN} B$ induces an
isomorphism in $H^*_{\mN}$.
\end{cor}

\begin{proof}
It suffices to prove the case when $B$ is the suspension spectrum of
$Z_+$ for  an unbased space $Z$. 
It is convenient to view this as the $\Gamma$ Borel cohomology of
$$\Gamma\times_{\mN}Z \lra \Gamma\times_{\mN} \Gamma\times_{\mN}Z\cong
\Gamma/\mN \times \Gamma\times_{\mN}Z. $$
Since the composite with projection is the identity, 
it therefore suffices to observe that by the lemma,  $\Gamma /\mN \lra *$ induces an isomorphism. 
\end{proof}

\subsection{Fixed points and induced spaces}
The purpose of this subsection is to show that the $L$-fixed point
spaces of induced spaces are made up of copies of induced spaces of
Weyl groups. 

More precisely, we suppose $L\subseteq \mT$ and consider its
conjugates inside $\mT$, this consists of the $\mW G$-orbit
of $L$, and we suppose the groups are $L=L_1, L_2, \ldots , L_c$ with
$L_i=L^{\gamma_i}$. In the usual way if $A$ is an $\mT$-space then 
$A^{L_i}=\gamma_i^{-1} (A^L)$. 

\begin{lemma}
\label{lem:fpind}
For a $\mT$-space $A$ we have
$$(G\times_{\mT}A)^L=\coprod_i W_G(L)\gamma_i \times_{\mT
  /L_i}A^{L_i}. $$
\end{lemma}

\begin{proof}
We note that the condition for $[g,a]$ to be $L$-fixed is that for
each $l\in L$ there is a $t\in \mT$ so that  $lg=gt$ and
$t^{-1}a=a$. The first condition determines $t$, so $[g,a]$ is only
fixed if $L^g\subseteq \mT$ and then   $a$ is fixed by $L^g$. Thus we obtain 
$$\coprod_i N_G(L)\gamma_i \times_{\mT }A^{L_i} \lra \coprod_i
N_G(L)/L \gamma_i \times_{\mT/L_i }A^{L_i} $$
as claimed. 
\end{proof}

\begin{cor}
\label{cor:fpNGamma}
For any $T$-space $A$, the map $\mN \times_{\mT}A \lra  G \times_{\mT}A $
induces an isomorphism of $W_GL$-equivariant Borel cohomology of
$L$-fixed points. 
\end{cor}

\begin{proof}
From Lemma \ref{lem:fpind}, we see that the map is a disjoint union of instances
of 
$$W_N(L)\times_{\mT /L}A^L \lra W_G(L)\times_{\mT /L}A^L. $$
This in turn is an instance of Corollary \ref{cor:NGamma} with $\Gamma=W_GL$.
\end{proof}

\subsection{Adjoint representations}
It is extremely interesting to see how the adjoint representation
behaves in moving between $G$ and $\mN$. Alternatively stated, 
this amounts to understanding the adjoint representation in  the Adams 
isomorphism. We write $L G$ for the adjoint representation, which is
to say the tangent space at the identity of $G$ with $G$ acting by
conjugation. For the torus $\mT$ there is also a rational version 
$L_{\Q}\mT=H_1(\mT ; \Q)$, so that there is a natural isomorphism
$L_{\Q}\mT \tensor \R\cong L\mT$. 

\begin{lemma}
\label{lem:suspLie}
We have a natural isomorphism
$H_*^{G}(X\sm S^{LG})=H_*^{\mN }(X\sm S^{L\mT})$. 
\end{lemma}

\begin{proof}[by stable equivariant formalism]
If $X$ is a finite free $G$-space then we have natural isomorphisms
$$\begin{array}{rcl}
H_*^{G}(\Sigma^{LG}X) 
&\cong &[S^0, X\sm H]^G_*\\
&\cong &[DX, H]^G_*\\
&\cong &H_G^* (DX)\\
&\cong &H_N^* (DX)\\
&\cong &[DX, H]^N_*\\
&\cong&[S^0, X\sm H]^{\mN}_*\\
&\cong&H_*^{\mN}(\Sigma^{L\mT }X) 
\end{array}$$
The two equivalences changing $X$ to $DX$ come from the formal
properties of duality. Since $X$ is finite and free, $DX$ is free,
giving the isomorphisms with Borel cohomology. The one relating
$G$-equivariant and $N$-equivariant Borel cohomology is Lemma
\ref{lem:class} (ii). The first and last isomorphisms are instances of
the Adams isomorphism. 
\end{proof}

\begin{proof}[by Lie group theory]
We observe directly that  $S^{L\mT}\lra S^{L  G}$ induces an
isomorphism in $H^*_{\mN}$, which is to say that multiplication by the Euler class of
$LG/LT$ is an isomorphism. More precisely, if $g=\dim G, t=\dim \mT$,  we show that 
the horizontals in the following diagram are isomorphisms of the $\mW G$-invariants 
$$\diagram 
\left[ H^*(B\mT)\tensor S^{LT}\right]^{\mW G} \rto \dto_=&\left[
  H^*(B\mT)\tensor S^{LG}\right]^{\mW G} \dto^=\\
\Sigma^t\left[ H^*(B\mT)\tensor H^t (S^{LT}) \right]^{\mW G}\rto &
\Sigma^g\left[ H^*(B\mT)\tensor H^g(S^{LG})\right]^{\mW G}
\enddiagram
$$
given by the multiplication by the product of the Euler classes of the
positive roots. 

We adjoin exterior variables to give a context, writing 
$A(V)=E(\Sigma V)\tensor P(\Sigma^2 V)$ for a vector space $V$.  
For an element $v\in V$ we write $\lambda (v)$ for the
corresponding element of $\Sigma V$ and $c(v)$ for the element of
$\Sigma^2V$. We consider the special case $V=L_{\Q}\mT$, so that 
$H^*(B\mT) =P(\Sigma^2V)$.  Thus
$$H^*_{\mT}(S^{LT}) \subseteq A(V)$$
consists of the $H^*(B\mT)$-submodule generated by $\det (\Sigma
V)$. Choosing an ordered basis $e_1, \ldots , e_r$ of $V$ we may let 
$\delta =\lambda (e_1)\sm \ldots \sm \lambda (e_r)$ be a generator of
$\det (\Sigma V)$. 
Now consider the adjoint representation of $G$ and choose a set $R_+$
of positive roots.  If we take $\kappa =\prod_{\alpha \in
  R_+}c(\alpha)$ then $\delta \kappa$ is the Euler class of $LG$.

The result is now Solomon's Lemma \cite{solomon}, but perhaps it is
illuminating   to sketch the proof in this case. We observe that $\delta \kappa$ is $\mW G $
invariant. Indeed, associated to $R_+$ there is the Weyl chamber on
which the roots are positive and  $\mW G$ is generated by reflections 
$s_{\alpha}$ in the walls of the Weyl chamber. Since $s_{\alpha}$ is a
reflection  $s_{\alpha}\delta =-\delta$.  On the other hand
$s_{\alpha}$  
negates $\alpha$ and permutes the other positive roots
(\cite[4.10]{BtD}). Hence $s_{\alpha}$ fixes $\delta \kappa$.

Since $H^*(BG)=H^*(B\mT )^{\mW G}$ it follows that 
$$H^*_G(S^{LG})=H^*(BG)\cdot \delta \kappa \subseteq
H^*_{\mT}(S^{L\mT})^{\mW G} .$$
Now we argue that any element $\delta f$ of the invariants is divisible by each
$c(\alpha)$. Since $\det (s_{\alpha})=-1$, we find
$f(s_{\alpha}c(v))=-f(c(v))$ for each $v$. Accordingly, for each $v$ in the
reflecting hyperplane $f(c(v))=0$. If we choose a basis consisting of
$\alpha$ together with elements of the reflecting hyperplane we see
$c(\alpha)$ divides $f$. Since any pair of positive roots are linearly
independent, it follows that $f$ is divisible by $\kappa $ as
required. 
\end{proof}

\subsection{The dual of Borel cohomology }
\label{subsec:Borel}
We let $b$ denote the representing $G$-spectrum for Borel cohomology,
so that, by definition,  
$$b_G^*(X)=H^*(EG_+\sm_G X) =[EG_+\sm X , H]_G^*=[X, F(EG_+, H)]_G^*.$$
This shows the representing spectrum is given by 
$$b=F(EG_+, H).$$
The associated homology theory is defined by 
$$b^G_*(X)=[S^0, X \sm b]^G_*. $$
The canonical warning is that this is not homology of the Borel
construction. Instead, we have  
$$b^G_*(X)=\colim_{\alpha}b^G_*(X_{\alpha})$$
where $X$ is the directed colimit of finite subspectra $X_{\alpha}$. 
For finite spectra $Y$ we have
$$b^G_*(Y)=[S^0, Y\sm b]^G_*=[DY, b]_G^*=b_G^*(DY) =H^*(EG_+\sm_G DY)
.$$
\begin{remark}
This calculation can be viewed as one of the motivations for
Borel-Moore homology, according to which  $b^G_*(X)$ would be the Borel-Moore
homology associated to Borel cohomology. However, since the essence of
Borel-Moore homology is really the use of locally finite chains it
would be misleading to call this Borel Borel-Moore homology.
\end{remark}

We will need a standard observation. 
\begin{lemma}
For finite $G$-spectra
$Y$ we have $b \sm Y \simeq *$ if and only if $b\sm DY \simeq *$. 
\end{lemma}

\begin{proof}
Since $b$ is a ring $G$-spectrum it follows that if $b\sm
Y\simeq *$ then $F(Y, b)\simeq *$. 
\end{proof}

Our main use of this homology theory is to formulate approppriate
analogues of Lemma \ref{lem:class}.
 
\begin{lemma}
\label{lem:bGWG}
Suppose $\Gamma$ is a compact Lie group with maximal torus $T$ and 
Weyl group $\mW \Gamma$ and that the order of $\mW \Gamma$ is
invertible in the coefficients. 
For $\Gamma$-spectra $A$,
there is a natural isomorphism
$$b^{\Gamma}_*(A)= \left[ b^{T}_*(A) \right]^{\mW \Gamma}$$
\end{lemma}

\begin{proof}
The forgetful map 
$$[S^0, b\sm A]^{\Gamma}\lra [S^0, b\sm A]^T$$
supplies a natural transformation 
$$b^{\Gamma}_*(A) \lra \left[ b^{T}_*(A) \right]^{\mW \Gamma}. $$
Since the order of $\mW \Gamma$ is invertible, both terms are homology
theories, and both preserve filtered colimits. It is an isomorphism
for finite complexes by Lemma \ref{lem:class} (iii). 
\end{proof}

 \section{The functor from $G$-spectra to $\protect \cA (G,toral)$}
\label{sec:GspectratoAG}

We have built a model of toral $G$-spectra by comparison with the
model for $\mT$-spectra. In this section, we elucidate the
relationship between these two models and thereby construct the functor $\piAG_*$
from $G$-spectra to $\cA (G,toral)$.

\subsection{Equivariance}
We have seen that rational $\mT$-spectra are modelled by
$\cA (\mT)$ and that there is a functor 
$$\piA_*: \mbox{$\mT$-spectra}\lra \cA (\mT)$$
defined by 
$$\piA_*(X)(L)=\pi^{\mT/L}_*(DE\mT /L_+\sm \Phi^LX)$$
and for finite $X$ this is $H^*_{\mT /L}(D\Phi^LX)$. 
The image of restriction from $G$-spectra to $\mT$-spectra has additional structure. To
start  with, we know that $\cA (\mT)$ admits an action of $\mW G$.

\begin{lemma}
\label{lem:piAisequiv}
The image of the composite
$$\mbox{$G$-spectra}\lra \mbox{$\mT$-spectra}\lra \cA (\mT) $$
consists of $\mW G$-equivariant modules and $\mW G$-equivariant maps. 
Accordingly, we have a functor 
$$\piA_*: \mbox{$G$-spectra}\lra \cA (\mT )[\mW G]=\cA (\mN , toral).$$
\end{lemma}

\begin{proof}
By definition 
$$\piA_*(X)(K\supseteq L) =\pi^{\mT /L}_*(\siftyV{K/L} \sm DE\cF /L_+ \sm
\Phi^LX). $$
The action of $\mW G$ is through conjugation by group elements. This
gives group homomorphisms $L\lra L^w$, and homeomorphisms between the
spaces corresponding to the groups. The homeomorphisms are equivariant for the
group homomorphism. 

The identification $\cA (\mN, toral)=\cA (\mT )[\mW G]$ is given in
Lemma \ref{lem:ANtoral}.
\end{proof}

\subsection{Restriction for free spectra}
In preparation for  explaining how restriction from $G$-spectra to
$\mN$-spectra is modelled, we  consider the inclusion 
$i: H \lra G$ of a subgroup. We have left  and right adjoints to restriction:
$$\adjointtriple{\mbox{$G$-spectra}}{i_*}{i^*}{i_!}{\mbox{$H$-spectra}}. $$

It is helpful to think first about what happens for free
spectra. We summarise the discussion from \cite{hg}.  Starting in the case when $G$
and $H$ are connected, we have a map 
$$\theta =i^*: H^*(BG)\lra H^*(BH). $$
This induces restriction of scalars $\theta^*$ which itself has left
and right adjoints
$$\adjointtriple{\mbox{$H^*(BH)$-mod}}{\theta_*}{\theta^*}{\theta_!}{\mbox{$H^*(BG)$-mod}}. $$
It is apparent that the two triples of adjoint functors cannot match up. It turns out that (when we
use the  the Eilenberg-Moore equivalence) it is $i_!$ that is modelled by
$\theta^*$, so that $i^*$ is modelled by $\theta_*$. 

The relevant analogy for us does not involve connected groups, so we
recall the general case from \cite{hg}.  We write $i_e: H_e \lra
G_e$ for the inclusion of the identity component, and $i_d: H_d\lra
G_d$ for the induced map on discrete quotients (not usually
injective).  In algebra, we again let
$\theta_e=i_e^*: H^*(BG_e)\lra H^*(BH_e)$ for the induced map in
cohomology. The main piece of data is $\theta =(\theta_e, i_d)$.
 It turns out that $i_!$ is modelled by a functor we call $\theta^*$, 
which is defined on
$H^*(BH_e)[H_d]$-modules $N$ by 
$$\theta^*(N)=\Hom_{\Q [H_d]}(\Q [G_d ], N)$$
(we note this is consistent with the previous notation when $H_d=G_d=1$).
Restriction of groups $i^*$ is then modelled by the functor $\theta_*$
left adjoint to $\theta^*$, which is 
defined on $H^*(BG_e)[G_d]$-modules $M$  by 
$$\theta_*(M) =H^*(BH_e)\tensor_{H^*(BG_e)}M. $$
Induction of spectra $i_*$ is then modelled by the functor
$\theta^{\dagger}$ left adjoint to $\theta_*$
defined on $H^*(BH_e)[H_d]$-modules $N$  by 
$$\theta^{\dagger}(N) =\Q [G_d] \tensor_{\Q [H_d]} \mathbb{D}
(G_e|H_e)\tensor_{H^*(BH_e)}N,   $$
where the relative dualizing module is defined by
$$\mathbb{D} (G_e|H_e)= \Hom_{H^*(BG_e)}(H^*(BH_e), H^*(BG_e)) .$$
In our case the relative dualizing module satisfies
$$\mathbb{D} (G_e|H_e)= \Hom_{H^*(BG_e)}(H^*(BH_e),
H^*(BG_e))\simeq \Sigma^{LG/H}H^*(BH_e) $$
and we may therefore  simplify the expression for $\theta^{\dagger}$  to find
$$\theta^{\dagger}(N) =\Q [G_d] \tensor_{\Q [H_d]} \Sigma^{LG/H}N.  $$

In the special case $H=\mN$ we note that $H_e=\mT$ and $H_d=\mW G$. In
view of the fact that $\mW G/\mW (G_e)\cong G_d$ we see that
$$ i_! \mbox{  is modelled by } 
\theta^*N=N^{\mW G_e} $$
$$ i^* \mbox{  is modelled by } 
\theta_*M=H^*(B\mT )\tensor_{H^*(BG_e)}M $$
and 
$$ i_* \mbox{  is modelled by } 
\theta^{\dagger}N= (\Sigma^{L G/\mT}N)_{\mW (G_e)} $$

\subsection{The image of a spectrum in the model}

We now make explicit the functor we use to relate
$G$-spectra to $\cA (G,toral)$. The motivation is that restriction to
the maximal torus is homotopically faithful, but the special form of
the objects in  the image mean that we can pass to invariants without
losing information. 

We will need to consider the functor 
$$\cA (\mN , toral)=\cA (\mT )[\mW G]\stackrel{\Psi}\lra \cA
(G,toral)$$
from Proposition \ref{prop:Psitheta}, where we use the Lie group
component structure of Subsection \ref{subsec:Liecomponent}. We recall that it was shown
to have  left adjoint $\theta_*$ defined by 
$$\theta_*(Y)= \RRa\tensor_{\RRinv}Y, $$
which on subgroups $K\subseteq \mT$ is 
$$(\theta_*Y)(K)= H^*(B\mT/K)\tensor_{H^*(BW_G^eK)}Y (K), $$

\begin{remark}
In view of the isomorphism $\Q [\mW G]^{\mW W_G^eK}=\Q [W_G^dK]$, the relationship between the two notations is 
$$(\Psi X)(K)=X(K)^{\mW W_G^eK}=\Hom_{(\mW G)_K}(\Q W_G^dK, X(K))=(\theta^*X)(K) .$$
\end{remark}

\begin{defn}
\label{defn:piAG}
The functor $\piAG_*: \mbox{$G$-spectra}\lra \cA (G,toral)$ is defined
as the illustrated composite of three functors:
$$\diagram 
\mbox{$G$-spectra}\rto^{\res^G_{\mN}} \dto_{\piAG_*}&\mbox{$\mN$-spectra}\dto^{\piA_*} \\
 \cA (G,toral) & \cA (\mT )[\mW G]\lto^{\Psi}
\enddiagram$$
\end{defn}

\begin{remark}
We note that specializing the definition to the case $G=\mN$ gives $\piAN_*=\piAT_*=\piA_*$, which is
consistent according to Lemma \ref{lem:piAisequiv}. 
\end{remark}

We immediately express $\piAG_*$ more directly in terms of
$G$-equivariant data. 

\begin{prop}
\label{prop:piAG}
For any $G$-spectrum $X$, and any subgroup $K\subseteq \mT$, we have
$$\piAG_*(X)(K)=b^{W_G^eK}_*(\Phi^KX). $$
If $X$ is a finite $G$-spectrum,  we can express this directly in terms
of Borel cohomology of fixed points of the dual
$$\piAG_*(X)(K)=H^*_{W_G^eK}(\Phi^K(DX)). $$
\end{prop}

\begin{remark}
It would be possible to give the statement of the proposition as the
{\em definition} of $\piAG_*(X)$. We used Definition \ref{defn:piAG}
instead because the deduction of the proposition from the definition
is a little more elementary than the reverse deduction. Indeed, if $\Gamma$ is a connected
group (such as $W_G^eK$) with maximal torus $T$ and $A$ is a $\Gamma$-space (which might have
arisen as $\Phi^K DX$ in some cases)  Lemma \ref{lem:class} gives the two formulae
$$H^*_{\Gamma}(A) =H^*_{\mT}(A)^{\mW \Gamma} $$
and
$$H^*_{\mT}(A)=H^*(BT)\tensor_{H*(B\Gamma)}
H^*_{\Gamma}(A). $$
We view the first as more elementary than the second.
\end{remark}

\begin{proof}
By definition 
$$\piAG_*(X)(K)=\pi_*^{\mT /K}(DE\mT/K_+\sm
  \Phi^KX)=b^{\mT/K}_*(\Phi^KX). $$
The result now follows by applying  Lemma \ref{lem:bGWG} with $\Gamma
=W_G^eG$ and  $A=\Phi^KX$. 
\end{proof}

\subsection{Restriction}
As in the case of free spectra, it will emerge that $\theta^*=\Psi$
corresponds to coinduction, and its left adjoint $\theta_*$
corresponds to restriction.

\begin{prop}
\label{prop:modelofresGN}
The following diagram commutes
$$\diagram
\mbox{toral-$G$-spectra}\rto^{\res^G_{\mN}} \dto^{\piAG_*} &\mbox{toral-$\mN$-spectra}\rto^{\res^{\mN}_{\mT}} \dto^{\piAN_*}
&\mbox{$\mT$-spectra} \dto^{\piAT_*} \\
\cA (G,toral)\rto^{\theta_*} &\cA (\mN,toral)\rto \dto^= &\cA (\mT)\dto^=\\
&\cA (\mT)[\mW G]\rto^{\res^{\mW G}_1}  &\cA (\mT)
\enddiagram$$
\end{prop}

\begin{proof}
The right hand two squares commute by the definition of $\piA_*$
together with 
Lemmas \ref{lem:piAisequiv} and \ref{lem:ANtoral}. 

By definition $\piAG_* X=\Psi \piAN_* (\res^G_{\mN}X)$, so the commutation of
the left hand square is given by the Proposition
\ref{prop:thetaPsionG} below. 
\end{proof}

\begin{prop}
\label{prop:thetaPsionG}
If $X$ is a $G$-spectrum then the counit 
$$\theta_* \piAG_*(X)=\theta_* \Psi \piAN_* (\res^G_{\mN} X)\stackrel{\cong}\lra
\piAN_*(\res^G_{\mN} X)$$
is an isomorphism. 
\end{prop}

\begin{remark}
In essence this amounts to two classical statements about Borel
cohomology (Lemma \ref{lem:class} (iii) and (iv)). 
\end{remark}

\begin{proof}
We consider the situation at $K\subseteq \mT$, for a $G$-spectrum $X$,  where we
have the map
$$H^*(B\mT /K)\tensor_{H^*(BW_G^eK)}\pi^{\mT/K}_*(DE\mT /K_+\sm
\Phi^KX)^{\mW W_G^eK}\lra  \pi^{\mT/K}_*(DE\mT /K_+\sm \Phi^KX). $$
Since $H^*(B\mT /K)$ is free over  $H^*(BW_G^eK)$, both sides commute
with direct limits in $X$, so  it suffices to prove this is an
equivalence for finite $X$, and these may be taken to be of the form
$DY$ for a finite spectrum $Y$. Since $\Phi^KDY\simeq D\Phi^KY$ for finite $Y$, and since 
$$DE\mT/K_+\sm D\Phi^KY\simeq D(E\mT/K_+ \sm \Phi^KY)$$ 
we may translate this into a statement about Borel cohomology of
the $WK$-spectrum $\Phi^KY$: 
$$H^*(B\mT /K)\tensor_{H^*(BW_G^eK)} H^*_{\mT/K}(Z)^{\mW W_G^eK}
\stackrel{\cong}\lra H^*_{\mT /K}(Z). $$
We note further that this only depends on the identity component
$W_G^eK$ of $WK$, and it is sufficient to consider the special case when 
$Z$ is free and the suspension spectrum of a space. 

The required isomorphism is then the special case $\Gamma =W_G^eK$ of
the Eilenberg-Moore theorem as in Lemma \ref{lem:class} (iv). 
This completes the proof of the proposition.
\end{proof}

We note that Proposition \ref{prop:thetaPsionG} has significant consequences:
only modules of the form $\theta_* N$ can
be $\piAN_*X$ for a $G$-spectrum $X$. 

\begin{example}
If $G=SO(3)$ we have $\mN =O(2)$ and $\mT =SO(2)$. Thus $H^*(B\mT)=\Q
[c]$ for $c$ of degree $-2$ with $W=O(2)/SO(2)$ acting as $-1$ on $c$,
and $H^*(BG) =H^*(B\mT)^W=\Q [d]$ where $d=c^2$ is of degree $-4$. 
 
We thus find that the only $\Q[c][W]$-modules occurring as the $\mT$-equivariant
homotopy of a free $G$-spectrum are those of the form
$M=\Q[c]\tensor_{\Q [d]} N$. In particular the eigenspaces of $+1$ and
$-1$ are related by 
$$M^-= c\cdot N =\Sigma^{-2} N = \Sigma^{-2}M^+.$$
For example $\Q [c]/(c^2)=\Q \oplus \Sigma^{-2}\tilde{\Q}$ occurs, but
the dual module  $\Q \oplus \Sigma^{2}\tilde{\Q}$ does not. 
\end{example}

Proposition \ref{prop:thetaPsionG}  gives the beginning of our main change of groups theorem.
\begin{cor}
If $X$ and $Y$ are $G$-spectra
then 
$$\Hom_{\cA (G,toral)}(\piAG_*X,\piAG_*Y)=\Hom_{\cA (\mT)}(\piAT_*X,
\piAT_*Y)^{\mW G}. \qqed $$
\end{cor}

\subsection{Coinduction}
We have just shown that $\theta_*$ models restriction. If 
the algebraic and topological categories were equivalent, it would
follow that the right adjoint of $\theta_*$ (viz $\Psi$) modelled the
right adjoint of restriction (viz coinduction). We show that this
expected relationship does indeed hold.

\begin{prop}
\label{prop:modelofcoindNG}
For any $\mN$-spectrum $Y$, 
$$\piAG_*(F_{\mN}(G_+, Y))=\Psi \piAN_*(Y), $$
so that the following diagram commutes
$$\diagram
\mbox{toral-$G$-spectra} \dto_{\piAG_*}
&\mbox{toral-$\mN$-spectra}\lto_{F_{\mN}(G_+, \cdot)} \dto^{\piAN_*}\\
\cA (G,toral) &\cA (\mN,toral)\lto^{\Psi} \ar@{=}[d] \\
&\cA (\mT)[\mW G]
\enddiagram$$
\end{prop}

\begin{remark}
In essence this amounts to a classical statement about Borel
cohomology (Corollary \ref{cor:NGamma}).
\end{remark}

\begin{proof}
First we note there is a natural transformation. Indeed, we may apply
$\Psi \piAN_*$ to the counit
$$\res^G_{\mN} F_{\mN}(G_+, Y)) \lra Y$$ 
to obtain a natural map
$$\piAG_*(F_{\mN}(G_+, Y))=\Psi (\piAN_* (\res^G_{\mN} F_{\mN}(G_+, Y)))
\lra  \Psi (\piAN_*(Y)). $$
Both of these are cohomology theories in the toral $\mN$-spectrum $Y$, so it
suffices to show that is is an equivalence when $Y=D\mN /K_+$
for $K\subseteq \mT$.  Thus we need only check that $\Psi \piAN_*$
vanishes on the cofibre of $\mN/K_+\lra G/K_+$, which was
Corollary \ref{cor:fpNGamma}. 
\end{proof}

\section{An Adams spectral sequence}
\label{sec:ASS}

We need to set up a means of calculation, so we will construct an
Adams spectral sequence based on $\cA (G,toral)$. We summarise the
method here, referring to the appropriate sections for proofs.

\subsection{Overview}
The main theorem of the paper is as follows. 

\begin{thm}
\label{thm:ASS}
There is an Adams spectral sequence for calculating maps between toral
$G$-spectra. For arbitrary rational toral $G$-spectra $X$ and $Y$ there is a
strongly convergent spectral sequence
$$E_2^{s,t}=\Ext_{\cA (G,toral)}^{s,t}(\piAG_*(X),
\piAG_*(Y))\Rightarrow [X,Y]^G_{t-s}.$$
The $E_2$ page lies between the $s=0$ line and the $s=r$ line, where
$r$ is the rank of $G$, so the spectral sequence collapses at the 
$E_{r+1}$-page. 
\end{thm}

\begin{proof}
We outline the standard strategy and deal with the main points in
succession. 

First, Proposition \ref{prop:AGtoralinj} shows that 
 the abelian category $\cA (G,toral)$ has enough injectives.

Accordingly, we may form an injective resolution 
$$0\lra \piAG_*(Y)\lra I_0\lra I_{1} \lra \cdots $$
of $\piAG_*(Y)$ in $\cA (G,toral)$. 

We then show that this can be realized by toral spectra. First the objects.

\begin{lemma}
\label{lem:realinj}
Enough injectives are realizable: there are enough injectives $I$ in
$\cA (G,toral)$ for which there exist toral $G$-spectra $\bbI$ with $\piAG_*(\bbI)=I$.
\end{lemma}

This is proved in Section \ref{sec:objects}.  

Next we show that maps between the injectives are realizable. 

\begin{prop}
\label{prop:mapstoinj}
If $\bbI$ is one of the injectives constructed in the proof of Lemma
\ref{lem:realinj}, then we have an isomorphism 
$$\piAG_* :[X, \bbI ]^G\lra \Hom_{\cA (G, toral)} (\piAG_*(X),
\piAG_*(\bbI)) =\Hom_{\cA (G, toral)} (\piAG_*(X), I). $$
\end{prop}

This is proved in Section \ref{sec:mapstoinj}. 

This enables us to construct an Adams tower
$$\diagram 
Y \ar@{=}[r]&Y_0 \dto &Y_1 \dto \lto &Y_2 \dto \lto &Y_3 \dto \lto& .\lto\\
&\bbI_0 &\Sigma^{-1}\bbI_1 &\Sigma^{-2}\bbI_2 &\Sigma^{-3}\bbI_3 &
\enddiagram$$

The construction starts by using Lemma \ref{lem:realinj} to realize
$I_0$ by a $G$-spectrum $\bbI_0$ and then Proposition \ref{prop:mapstoinj} to realize
$\piA_*(Y)\lra I_0$ by a map $Y\lra \bbI_0$. We now take  $Y_1$ to be
its fibre so that  $\piAG_*(\Sigma Y_1)=\cok (\piAG_*(Y)\lra I_0)$. We
may now repeat, using Lemma \ref{lem:realinj} to realize $I_1$ and 
 Proposition \ref{prop:mapstoinj} to give  a map $Y_1 \lra
\Sigma^{-1}\bbI_1$ realizing the map in the algebraic
resolution. Higher Adams covers are constructed by continuing this
process. 

This process terminates by Proposition \ref{prop:AGtoralinj}, which
shows the category $\cA (G, toral)$ has finite injective dimension.

We deduce that the Adams tower stops at $Y_{r+1}$ with
$\piAG_*(Y_{r+1})=0$. Applying $[X,\cdot ]^{G}$ to the tower we obtain
a spectral sequence. By Proposition \ref{prop:mapstoinj} it has the stated
$E_2$ term. 

The convergence statement is as follows. 

\begin{lemma}
\label{lem:convergence}
If $X$ is a toral $G$-spectrum with $\piAG_*(X)=0$ then $X\simeq *$. 
\end{lemma}


\begin{proof}
Suppose then
that $\piAG_*(X)=0$, and we want to prove that $X$ is contractible. By
Proposition \ref{prop:toraldetection} it suffices to show
$\piAT_*(X)=0$. By definition,  $\piAG_*(X)=\Psi \piAT_*(X)$, so the
result follows, since by Proposition \ref{prop:thetaPsionG} we have
$$\piAN_*(X)=\theta_*\Psi \piAN_*X=\theta_* \piAG_*X.$$ 
\end{proof}

Modulo the deferred proofs of the lemmas, this completes the proof of
Theorem \ref{thm:ASS}. 
\end{proof}

\section{Realizing enough injectives}
\label{sec:objects}
In this subsection we prove Lemma \ref{lem:realinj} by realizing
enough of  the injectives described in Section \ref{sec:AGtoralinj}. 

\subsection{Supports}
For a commutative Noetherian ring, the indecomposable injectives
correspond to the prime ideals, and the injective corresponding to a
prime $\wp$ is the injective hull of the residue field of $\wp$. The
support of a sum of these is the collection of primes involved. 
The same  principle applies in our context. We have notions of algebraic
and geometric injectives and in both cases the support is a set of closed subgroups. 

In $\cA (\mT )$ the support is given by the maximal subgroup on which a module
is non-zero. This means that the primes correspond to closed subgroups $K$, the ring
corresponding  to $K$ is $H^*(B\mT /K)$ with residue field $\Q$ and 
injective hull $H_*(B\mT /K)$. To obtain the corresponding object of 
$\cA (\mT)$, we apply the functor $f_K^{\mT}$ right adjoint to
evaluation at $K$.

Moving from $\mT$ to $\mN$, we saw in Section \ref{sec:AGtoralinj}, 
that the same idea works for $\cA (\mN , toral)=\cA (\mT)[\mW G]$
provided we use the complete $\mW G$ orbit $(K)$ rather than 
the singleton $K$. For $G$, the support is detected through
restriction to $\mN$.

\subsection{Some idempotent spaces}
The support in the topological setting corresponds to geometric
isotropy.  Indecomposable injectives are realized by the simplest possible 
space with geometric isotropy equal to the support. We pause to
catalogue some of these spaces. 

The geometric isotropy 
$$\GI (X)=\{ K \st \Phi^KX \not
\simeq_1 *\}$$
consists of subgroups where the geometric fixed points are
non-equivariantly essential. We further restrict to spectra where the
geometric fixed points are nonequivariantly either $S^0$ or contractible, which we might
call `locally idempotent'.

We recall that a collection $\mcH$ of subgroups closed under conjugacy
is called a 
{\em
 family} if it is closed under passage to subgroups, it is called a 
{\em
 cofamily} if it is closed under passage to supergroups, and it is
called
an {\em interval} if it contains any subgroup $K$ which lies between
two elements of $\mcH$. Intervals of subgroups are precisely those
collections which are  the intersection of a family and a cofamily.

\begin{defn}
If $\mcH$ is an interval of subgroups we write $\Lambda (\mcH)$ for
the set of subgroups of elements of $\mcH$ (which is the smallest
family containing $\mcH$) and 
$V (\mcH)$ for the set of supergroups of elements of $\mcH$ (which is the smallest
cofamily containing $\mcH$) and we define
$$E\langle \mcH\rangle := E\Lambda (\mcH)_+\sm \tilde{E}(All \setminus V(\mcH)).$$
\end{defn}

The proof of the following lemma is immediate from the Geometric Fixed
Point Whitehead Theorem. 

\begin{lemma}
If $\mcH$ is an interval, and we choose a family $\cF$ of subgroups and a
cofamily $\cC$ of  subgroups so that $\mcH =\cF \cap \cC$ then 
$$E\langle \mcH\rangle \simeq E\cF_+ \sm \tilde{E}(All \setminus
\cC). $$
The space  $E\langle \mcH\rangle $  is an idempotent spectrum with geometric isotropy $\mcH$, 
and any other locally idempotent spectrum with geometric isotropy
$\mcH$ is equivalent to it. \qqed
\end{lemma}

\begin{remark}
It is worth recording the following easy observations.
\begin{enumerate}
\item $\GI (E\langle \mcH \rangle)=\mcH .$
\item If $\cF$ is a family then 
$$E\langle \cF \rangle =E\cF_+, $$
\item If $\cC$ is a cofamily then 
$$E\langle \cC \rangle =\tilde{E}(All \setminus \cC)$$
\item Given two intervals $\mcH_1$ and $\mcH_2$ we have an equivalence
$$E\langle \mcH_1 \rangle \sm  E\langle \mcH_2 \rangle \simeq  E\langle \mcH_1
\cap \mcH_2 \rangle . $$
\item  If $K$ is a subgroup of $G$ and  $\mcH$ is an interval of subgroups of
$G$, we may consider the interval $\mcH|_K$ of subgroups of $K$ from 
 $\mcH$ and  then  
$$\res^G_KE_G\langle \mcH \rangle =E_K\langle \mcH|_K \rangle .  $$ 
\end{enumerate}
\end{remark}

\subsection{Idempotent spaces from conjugacy classes}
\label{subsec:idempconj}
We apply the generalities in our standard context with $G$ a compact
Lie  group with  maximal torus $\mT$ and $\mN =N_G(\mT)$. 

The spectra we are concerned with are idempotent spectra with all 
the geometric isotropy groups coming from a single
conjugacy class in a larger group. The point of the previous
subsection was  to point out that in this case the geometric isotropy determines the
object. This subsection records some immediate consequences for single
conjugacy classes. 

For the interval $(K)_G$ we consider the space
$$\elrG{K}=E\Lambda_G (K)_+\sm \Et (All \setminus V_G (K))$$
where $\Lambda_G (K)$ is the family of subgroups $G$-subconjugate to $K$ and
$V_G (K)$ is the cofamily of subgroups containing a $G$-conjugate  of $K$. In the
following, it is helpful to introduce some temporary notation. We write $P=N_{\mN}K$ for the subgroup of
$\mN$ fixing $K$, and we suppose the $\mN$ conjugacy class  of $K$  is $\{ K_1,
\ldots, K_s\}$, so that $s=|\mN : P|$.

\begin{lemma}
\label{lem:resegk}
There is an equivalence of $\mT$-spectra
$$\res^{P}_{\mT}\elrP{K}\simeq \elrT{K}.$$
There is an equivalence of $\mN$-spectra
$$e_{(\mT)}\res^G_{\mN}\elrG{K}\simeq \elrN{K}\simeq \mN_+
\sm_{P} \elrT{K},  $$
and hence an equivalence of $\mT$-spectra
$$\res^G_{\mT}\elrG{K}\simeq \bigvee_{i=1}^s \elrT{K_i}. $$
\end{lemma}

\begin{remark}
The idempotent in the second statement  is necessary. Consider the special case of $G=SO(3)$,
where $\mN =O(2)$ and $\mT= SO(2)$. The dihedral group of order 2
in $O(2)$ is not conjugate in $O(2)$ to a subgroup of $\mT$, but in
$SO(3)$ it is. 
\end{remark}
\begin{proof}
The first equivalence is clear. 

Two subgroups of $\mT$ which are conjugate in $G$ are conjugate in $\mN$ (the proof for elements in
\cite[IV.2.5]{BtD} applies to cover non-cyclic subgroups of
$\mT$). The geometric isotropy of $\elrG{K}$ is the single conjugacy
class $(K)_G$. The part lying in $\mT$ is the $\mN$-conjugacy class. 

Now there is a natural map of $\mN \cap N_G(K)$-spaces $\elrT{K}\lra
\elrG{K}$ which is $\{ K\} \lra (K)_{\mN }$ on supports. Since $\mT$
centralises $K$ this extends to $\mN \times_{\mN \cap N_G(K)}\{K\}\cong
(K)_{\mN}$. 
\end{proof}

The interaction with coinduction is also important. The point to note is
that in coinducing from $E_{\mT}\lr{K}$ there are three significant
stopping points: $P=N_{\mN}K$ (since $\{K \}=(K)_{\mT}=(K)_{P}$), $\mN$ and $G$.

\begin{lemma}
\label{lem:coindelrT}
If $K$ is a subgroup of $\mT$  and $P=N_{\mN}K$ then 
we have the following two equivalence of $P$-spectra
$$F_{\mT}(P_+ , E_{\mT}\lr{K})\simeq P/\mT_+\sm E_{P}\lr{K}.$$
and
$$F_{\mT}(\mN_+ , E_{\mT}\lr{K})\simeq P/\mT_+ \sm \elrN{K}.$$
\end{lemma}

\begin{proof}
The first statement is a standard untwisting result. 

For the second, we calculate
$$\begin{array}{rcl}
F_{\mT}(\mN_+, \elrT{K}) &\simeq &F_{P}(\mN_+,
F_{\mT}(P_+, \elrT{K})\\
&\simeq &F_{P}(\mN_+, P/\mT_+\sm \elrP{K})\\
&\simeq &\mN_+ \sm_P P/\mT_+ \sm \elrP{K}\\
&\simeq &P/\mT_+ \sm \mN_+ \sm_P \elrP{K}\\
&\simeq &P/\mT_+ \sm \elrN{K},\\
\end{array}$$
where the final equivalence comes from Lemma \ref{lem:resegk}.
\end{proof}

Coinducing up to $G$ has little effect. 

\begin{lemma}
\label{lem:coindenk}
There is an equivalence 
$$\elrG{K}\simeq F_{\mN}(G_+, e_{(\mT)}\elrG{K}) \simeq F_{\mN}(G_+,
\elrN{K}). \qqed$$
\end{lemma}

\subsection{Realizing injectives}
Again we rely on  \cite{tnq1}, which  shows  that in  $\cA (\mT)$ the
basic injective with support $K\subset \mT$ corresponds to the space $\elr{K}$.
More precisely, 
$$\piAT_*(E_{\mT}\lr{K})=f_K^{\mT}(H_*((B\mT/K)^{L\mT /K}))$$
where $f_K^{\mT}$ is right adjoint to evaluation at $K$ as before. 
Since we have now catalogued behaviour under change of groups in
algebra and topology, we can now read off the values we require. 

\begin{cor}
\label{cor:piAegk}
The images of $\elrG{K}$ in $\cA (N,toral)$ and $\cA (G,toral)$ are
given by the formulae
$$\piAN_*\elrG{K}= f^{\mN}_{(K)}(H_*((B\mT /K)^{L(\mT/K)}))$$
and 
$$\piAG_*\elrG{K}=f^G_{(K)}(H_*((BW_G^eK)^{L W_G^eK})), $$
where $f^{\mN}_{(K)}$ and $f_{(K)}^G$ are right adjoint to evaluation at $K$.
\end{cor}

\begin{proof}
We have constructed $\elrG{K}$ so that its geometric isotropy is
concentrated on $(K)_{\mN}$, so the module is concentrated on conjugates of
$K$ in $\mT$. In view of equivariance, we need only identify the value
at a single point in the orbit, and we find
$$\piAN_*(\elrG{K})(K) =\pi^{\mT /K}_*(DE\mT /K_+\sm E\mT /K_+)
=\pi^{\mT /K}_*(E\mT /K_+)=H_*(B\mT /K^{L (\mT /K}).$$

The second statement follows from the first using Lemma
\ref{lem:coindenk},  since  by Lemma \ref{lem:suspLie}
$$H_*((B\mT/K)^{L(\mT/K)})^{W_G^eK}\cong H_*((B\mT/K)^{L(\mT /K)})_{W_G^eK}\cong H_*(BW_G^eK)^{L(W_G^eK)}.$$
\end{proof}

We actually need slightly more general injectives, so that we can
embed  all representations  of $W_G^d(K)$. Of course there are many possible
choices. We could start from $\elrN{K}$ and coinduce,  but it turns
out that the proof is slightly streamlined by starting from
$\elrT{K}$. We give the calculations for both by way of comparison. 

\begin{cor}
\label{cor:piAegkplus}
The images of $\mN/\mT_+\sm \elrN{K}$ in $\cA (N,toral)$ and its
coinduced spectrum $F_{\mT}(G_+,  \elrN{K})$ in  
$\cA (G,toral)$ are given by the formulae
$$\piAN_*(\mN/\mT_+\sm \elrN{K})= f^{\mN}_{(K)}(\Q[ \mW G]\tensor H_*((B\mT /K)^{L(\mT/K)}))$$
and 
\begin{multline*}
\piAG_*(F_{\mT}(G_+, \elrN{K}))= \Psi f^{\mN}_{(K)}(\Q[ \mW
G]\tensor H_*((B\mT /K)^{L(\mT/K)}))\\
= f^{G}_{(K)}(\Q[ \mW G/WG_K^e]\tensor H_*((B\mT /K)^{L\mT /K})), 
\end{multline*}
where $f^{\mN}_{(K)}$  and $f^G_{(K)}$ are right adjoints to evaluation at $K$.
\end{cor}

\begin{proof}
The statement for $\mN$ follows easily from the previous corollary,
recalling from Subsection \ref{subsec:fK} that modules
over $(K)_{\mN}$ are determined from their value over $K$ by conjugation. 

The statement for $G$ follows since $\Psi$ models coinduction as in
Proposition \ref{prop:modelofcoindNG}.

We note that if $N$ is a $H^*(B\mT /K)[\mW G_K]$-module, there is a
natural transformation 
$$f^G_{(K)}(\Psi N)=\Psi f^{\mN}_{(K)}(\theta_*\Psi N)\lra \Psi f^{\mN}_{(K)}(N),  $$
when evaluated at $K$ the comparison is the identity
$$\Psi \theta_*\Psi N\lra \Psi N. $$
\end{proof}

The values that we will actually use in the proofs are as follows. 

\begin{cor}
\label{cor:piAcoind}
The images of the coinduction of $\elrT{K}$ to $\mN$-spectra and $G$-spectra  in the algebraic
categories is given by 
$\cA (G,toral)$ are given by the formulae
$$\piAN_*(F_{\mT}(\mN_+, \elrT{K})) = 
f^{\mN}_{(K)}(\Q[ (\mW G)_K]\tensor H_*((B\mT /K)^{L(\mT/K)}))$$
and 
\begin{multline*}
\piAG_*(F_{\mT}(G_+, \elrT{K}))= \Psi f^{\mN}_{(K)}(\Q[ (\mW
G)_K]\tensor H_*((B\mT /K)^{L(\mT/K)}))\\
= f^{G}_{(K)}(\Q[ W_G^dK]\tensor H_*((B\mT /K)^{L\mT /K})), 
\end{multline*}
where $f^{\mN}_{(K)}$  and $f^G_{(K)}$ are right adjoints to evaluation at $K$.
\end{cor}

\begin{proof}
The first statement follows from Lemma \ref{lem:coindelrT}, noting that
$(\mW G)_K=N_{\mN}K/\mT =P/\mT$. 

The second statement follows as in the proof of Corollary
\ref{cor:piAegkplus}. 
\end{proof}

\section{Maps into injectives}
\label{sec:mapstoinj}

In this section we give control over maps to realizable injectives by
proving Proposition \ref{prop:mapstoinj}. Since this is where we get control
over the maps in our category, it is perhaps not surprising that it is
the most delicate part of the argument. 

\begin{prop}
\label{prop:mapstoinjproved}
If $\bbI$ is a $G$-spectrum realizing one of the injectives $I$ constructed in the proof of Lemma
\ref{lem:realinj}, then we have an isomorphism 
$$\piAG_* :[X, \bbI ]^G\lra \Hom_{\cA (G, toral)} (\piAG_*(X),
\piAG_*(\bbI)) =\Hom_{\cA (G, toral)} (\piAG_*(X), I). $$
\end{prop}

\begin{proof}
Since $I$ is injective,  both sides are cohomology theories of $X$, it suffices to prove
the result for $X=G/K_+$ where $K$ is a subgroup of $\mT$. In fact we
will prove it more generally for $X=G_+\sm_{\mT}A$ for some finite
$\mT$-spectrum $A$. We consider
the diagram 
$$\diagram 
[  G_+\sm_{\mT}A, \bbI ]^G\ddto^{\cong} \rto^(0.35){\piAG_*}&
\Hom_{\cA (G,toral)}(\piAG_*(G_+\sm_{\mT}A), \piAG_*(\bbI))\dto^= \\
&
\Hom_{\cA (G,toral)}(\Psi \piAN_*(G_+\sm_{\mT}A), \Psi \piAN_*(\bbI))\dto^a \\
[\mN_+\sm_{\mT}A, \bbI ]^{\mN}\ddto^{\cong} \rto^(0.35){\piAN_*}&\Hom_{\cA (\mN)}(
\piAN_*\mN_+\sm_{\mT} A, \piAN_*(\bbI)) \dto^=  \\
&\Hom_{\cA (\mT)}( \piAN_*\mN_+\sm_{\mT} A, \piAN_*(\bbI))^{\mW G}\dto^b\\
[A, \bbI ]^{\mT} \rto^(0.35){\piAT_*}_(0.35){\cong}&\Hom_{\cA (\mT)}(
\piAT_* A, \piAT_*(\bbI)) 
\enddiagram$$

The bottom horizontal is an isomorphism from the $\mT$-equivariant
Adams spectral sequence of \cite{tnq1}, since $\piAT_*(\bbI)$ is
injective.  The two left hand vertical isomorphisms come from the
induction-restriction adjunction. The two right hand vertical
isomorphisms are definitions. 

It therefore remains to describe the maps $a$ and $b$ so that the
diagram commutes and to show that $a$ and $b$ are isomorphisms. 

We will deal with $b$ first, because it is straightforward. Since $\piAN_*=\piAT_*$ if we ignore the $\mW
G$-action, we may take $b$ to be induced by the $\mT$-map $\beta:
A\lra \mN \sm_{\mT}A$. The diagram commutes, since by definition the left hand
vertical factors through the forgetful map 
$$[\mN_+\sm_{\mT} A,
\bbI]^{\mN}\lra [\mN_+\sm_{\mT} A, \bbI]^{\mT}. $$
The fact that $b$ is an isomorphism follows from a lemma. 

\begin{lemma}
The map $\beta$ induces an isomorphism 
$$\piAN_*(\mN_+\sm_{\mT}A)=\mW G\tensor \piAT_*(A), $$
where the functor on the right is the induction functor left adjoint
to restriction. \qqed
\end{lemma}
For the map $a$ we use the diagram
$$\diagram 
[  G_+\sm_{\mT}A, \bbI ]^G\dto \rto^(0.35){\piAG_*}&
\Hom_{\cA (G,toral)}(\Psi \piAN_*(G_+\sm_{\mT}A), \Psi
\piAN_*(\bbI))\dto^{\theta_*}_{\cong} \\
[  G_+\sm_{\mT}A, \bbI ]^{\mN}\dto^{\alpha^*} \rto^(0.35){\piAN_*}&
\Hom_{\cA (\mT)}( \piAN_*(G_+\sm_{\mT}A), \piAN_*(\bbI))^{\mW G}\dto^{(\alpha_*)^*}\\
[\mN_+\sm_{\mT}A, \bbI ]^{\mN} \rto^(0.35){\piAN_*}&\Hom_{\cA (\mT)}(
\piAN_*\mN_+\sm_{\mT} A, \piAN_*(\bbI))^{\mW G}
\enddiagram$$
We have used the fact that the counit is an isomorphism on restrictions from $G$ (Proposition
\ref{prop:thetaPsionG}) to identify the codomain of $\theta_*$ and to
see it is an isomorphism. 
In short, $a$ comes from the map 
$$\alpha_*: \piAN_*(\mN_+\sm_{\mT}A) \lra \piAN_*(G_+\sm_{\mT}A)$$
induced by the $\mN$-map $\alpha: \mN_+\sm_{\mT}A \lra G_+\sm_{\mT}A. $

We will show that $(\alpha_*)^*$ is an isomorphism, but we pause to
observe that this is fairly subtle, since the map $\alpha_*$ itself is usually not an isomorphism. 

\begin{example}
Consider the special case $\bbI =EG_+$ we have
$$\diagram
[G_+\sm_{\mT}A, EG_+]^G\rto^(0.3){\cong} \dto^=&
\Hom_{H^*(BG_e)}(H_{G_e}^*(D(G_+\sm_{\mT} A)), H_*(BG_e^{LG}))^{G_d}\dto^= \\
[\mN_+\sm_{\mT}A, EG_+]^{\mN}\rto^(0.3){\cong} \dto^=&
\Hom_{H^*(B\mT )}(H_{\mT}^*(D(\mN_+\sm_{\mT} A)), H_*(B\mT^{LT}))^{\mW G} \dto^=\\
[A, EG_+]^{\mT}\rto^(0.3){\cong} &
\Hom_{H^*(B\mT )}(H_{\mT}^*(D(A)), H_*(B\mT^{L\mT}) ).
\enddiagram$$
The reader may find it instructive to think how the suspensions match
up. 

More specifically still, we may take $G=SO(3)$, $\mN =O(2)$ and $\mT
=SO(2)$, with $A=S^0$. Of course  $H^*_{\mT}(DA)=H^*B\mT$ so that we
see from the bottom right that the value is $\Q$ in each positive degree and 0 elsewhere.  
At the top left, we use the fact that $H^*_{\mT}(DG/\mT_+)$ is a copy of $\Q$ in codegree $-2$ and a copy of 
$\Q \mW G$ in even codegrees $\geq 0$, and its ring of $\mW
G$-invariants $H^*_{G_e}(DG/\mT_+)$ is a free $H^*(BG_e)$-module on
generators of cohomological degrees $0$ and $-2$. \qqed
\end{example}

To make further progress, it is convenient to make a specific choice for $\bbI$.
Indeed,  since $\piAN_*(G_+\sm_{\mT}A )$ and $\piAN_*(\mN_+\sm_{\mT} A)$ are small, it suffices
 to deal with the case $\bbI =F_{\mT}(G_+, \elrT{K})$ for some $K$. For any finite
 $\mN$-spectrum $B$ we have $\Phi^KDB=D\Phi^KB$ and Corollary
 \ref{cor:piAcoind} gives the value $\piAN_*(\bbI)$. Abbreviating $\cA (N,toral)$ to $\cA (N)$,  we may calculate
$$\hspace{-12ex}
\begin{array}{rcl}
\Hom_{\cA (\mN)} (\piAN_*(B), \piAN_*(F_{\mT}(G_+, \elrT{K})) )
&\cong&\Hom_{\cA (\mN)} (\piAN_*(B), \theta_*\Psi
\piAN_*(F_{\mT}(\mN_+, \elrT{K})) )\\
&\cong&\Hom_{\cA (\mN)} (\piAN_*(B), f^{\mN}_K(\theta_*\Psi H_*(B\mT
/K^{L\mT/K})[(\mW G)_K])  )\\
&\cong&\Hom_{H^*(B\mT/K)} (H^*_{\mT/K}(D\Phi^KB), \theta_*\Psi H_*(B\mT
/K^{L\mT/K})[(\mW G)_K])  )^{(\mW G)_K}\\
&\cong&\Hom_{H^*(B\mT/K)} (H^*_{\mT/K}(D\Phi^KB), \theta_*H_*(B\mT
/K^{L\mT/K})[W_G^dK])  )^{(\mW G)_K}\\
&\cong&\Hom_{H^*(B\mT/K)} (H^*_{\mT/K}(D\Phi^KB), \theta_*H_*(B\mT
/K^{L\mT/K}))  )^{(\mW G)_K^e} \\
&\cong&\Hom_{H^*(B\mT/K)} (H^*_{\mT/K}(D\Phi^KB^{L\mT /K}), \theta_*H_*(B\mT
/K))  )^{\mW W_G^eK} \\
\end{array}$$
As an $H^*(BW_G^eK)$-module $H_*(B\mT /K)$ is a sum of copies of
$H_*(BW_G^eK)$, and hence 
as an $H^*(B\mT /K)$-module $\theta_* H_*(B\mT /K)$ is a sum of copies of
$H_*(B\mT /K)$. The above functor is thus a sum of copies of 

$$\begin{array}{rcl}
\Hom_{H^*(B\mT/K)} (H^*_{\mT/K}(D\Phi^KB^{L\mT /K}), H_*(B\mT
/K))  )^{\mW W_G^eK} &\cong&\left[ H_*^{\mT/K}(D\Phi^KB^{L\mT
    /K}))\right]^{\mW W_G^eK} \\
&\cong&H_*^{W_G^eK}(D\Phi^KB^{LW_G^eK})) \\
\end{array}$$
where the final  isomorphism is Lemma \ref{lem:suspLie}. 

It suffices to show that $\alpha$ induces an isomorphism of this
functor of $B$, or equivalently that the functor vanishes on 
$$Q(A)=\cofibre (\mN_+\sm_{\mT}A\stackrel{\alpha}\lra G_+\sm_{\mT}A).$$
 Now the following groups vanish together
$$H^{W_G^e K}_*(D\Phi^K Q(A)^{L W_G^eK}), H_{W_G^e K}^*(D\Phi^K Q(A)^{L W_G^e K)}),
H_{W_G^e K}^*(\Phi^K Q(A)). $$
The first two are vector space duals, and the last two vanish together by the
standard observation about ring spectra recalled in Subsection
\ref{subsec:Borel}. 
 The result follows from Corollary \ref{cor:fpNGamma}. 
\end{proof}


\section{Essential surjectivity}
\label{sec:essepi}
We want to show that the functors $\piA_*$ are essentially
surjective, so that our modelling categories are no bigger than
necessary. 

\begin{lemma}
Every object of $\cA (G,toral)$ is realizable by a toral $G$
spectrum. 
\end{lemma}

\begin{proof} We may use the ingredients of the proof of the Adams
  spectral sequence. Suppose then that $M$ is a module in $\cA
  (G,toral)$. By Proposition \ref{prop:AGtoralinj}, this has an injective resolution 
$$0\lra M\lra I_0 \lra \cdots \lra I_r\lra 0. $$
We now set about constructing a toral $G$-spectrum $Y$ with
$\piAG_*Y=M$. When $Y$ is constructed, we will in retrospect see that
we have found the dual Adams tower $\{ Y^s\}$ where this is related to
the Adams tower by cofibre sequences $Y_s\lra Y\lra Y^s$. 

In any case, we construct a tower
$$\diagram
* \ar@{=}[r]& Y^0\dto &Y^1 \dto \lto &Y^2 \dto \lto&\cdots \lto  &Y^r \dto \lto &Y^{r+1}=Y \lto \\
&\Sigma^1\bbI_0&\bbI_1&\Sigma^{-1}\bbI_2&&\Sigma^{1-r}\bbI_r&\\
\enddiagram$$
For each $s$, the $G$-spectrum $\bbI_s$  is a realization of $I_s$, which exists 
 by Lemma \ref{lem:realinj}. We build the tower recursively, starting
 with $Y^0=*$ and $Y^1=\bbI_0$.  Supposing we have constructed the
 tower up to $Y^s$, we find an exact sequence
$$0\lra \Sigma^{1-s} C_{s+1}\lra \piAG_*Y_s \lra M\lra 0$$
in $\cA (G,toral)$,  where $C_{s+1}=\im (I_s\lra I_{s+1})$. Since $I_{s+1}$ is injective, we
may extend the map $C_{s+1}\lra I_{s+1}$ over $\piAG_*(Y_s)$ and then by
Proposition
\ref{prop:mapstoinj} we may realize this by a map $Y_s\lra
\Sigma^{1-s}\bbI_s$. We then take $Y^{s+1}$ to be the fibre,
completing the step. Since $C_{r+1}=0$, the process finishes in $r$
steps with $Y=Y^{r+1}$ having $\piAG_*(Y)=M$ as required. 
\end{proof}

\begin{remark}
For the special case $G=\mN$, one may work more directly from the case
of a torus. 
\end{remark}

\section{Change of groups}
\label{sec:hg}

We now suppose given a group $G$ and a subgroup $H$, and we choose
maximal tori $S$ of $G$ and $T$ of $H$ with  $S \supseteq T$. We note that it does not
follow that there is a containment of normalisers of maximal tori. 

Writing $i:H\lra G$ for the inclusion, the restriction map $i^*$ from
$G$-spectra to $H$-spectra has
left adjoint the induced spectrum $i_*Y=G_+\sm_HY$ and right adjoint
$i_!Y=F_H(G_+, Y)$ from $G$-spectra to
$H$-spectra.  Applying idempotents these  give functors on toral
spectra:
$$\adjointtriple{\mbox{$G$-spectra}}{i_*}{i^*}{i_!}{\mbox{$H$-spectra.}} $$

It is the purpose of this section to describe the algebraic
counterparts. If the ranks of the groups differ then there is only a
good story at the level of derived functors. The exposition will deal
with the general case, and simply note that if the ranks are equal
then the effect of using derived functors is nugatory.

The case of a torus is considerably simpler, and since
we will also reduce the general case to that of the torus, we will
deal with tori first in the next subsection. For the equal rank case
the content is vacuous, so readers interested only in equal rank can
skip Subsection \ref{subsec:tori}

\subsection{Tori}
\label{subsec:tori}
In this section we consider the case when $G=S$ and $H=T $ are
tori. We let $j: T\lra S$ denote the inclusion and $\lambda=j^*:
H^*(BS)\lra H^*(BT)$. To prove the assertion requires working with the
specific Quillen equivalences used in \cite{tnqcore}, so we will not
prove it here. On the other hand, special cases can be seen: free
spectra, and homologically simple objects. 

\begin{conj}
\label{conj:toralchange}
Given an inclusion $j: T\lra S$ of tori, the change of groups functors
$$\adjointtriple{\mbox{$S$-spectra}}{j_*}{j^*}{j_!}{\mbox{$T$-spectra}}$$
are modelled at the derived level by the functors
$$\adjointtriple{\mbox{\cA (S)}}{\lambda^!}{\lambda_*}{\lambda^*}{\mbox{\cA (T)}}.$$
For an object $M$ of $\cA (S)$, $\lambda_*M$ is defined on
subgroups $L\subseteq T$ by 
$$(\lambda_*M)(L)=H^*(BT/L)\tensor_{H^*(BS/L)} M(L), $$
where the tensor product is derived.  
For an object $N$ of $\cA (T)$, the objects $\lambda^*M$ and $\lambda^!M$ are
defined on subgroups $K\subseteq S$ by 
$$
(\lambda^*N)(K)=\dichotomy{N(K)& \mbox{ if $K\subseteq T$}}{0&\mbox{
   otherwise}}$$
and
$$
(\lambda^!N)(K)=\dichotomy{\Sigma^{LS/LT}N(K)& \mbox{ if $K\subseteq T$}}{0&\mbox{ otherwise}}$$
\end{conj}



\subsection{General case}

We now return to the general case when $S$ is the maximal torus of
$G$, $T$ is the maximal torus of $H$. We write  $i:H\lra G$ and
$j:T\lra S$ for the inclusions with induced maps
$$\theta=i^*: H^*(BG)\lra H^*(BH)$$
and 
$$\lambda=j^*: H^*(BS)\lra H^*(BT). $$
We will state the proposition in the equal rank case (i.e., when
$S=T$), but we have stated it so that it will hold at the derived
level in general provided Conjecture \ref{conj:toralchange} holds.

\begin{prop}
\label{prop:generalchange}
If $G$ and $H$ have the same rank, then the change of groups functors
$$\adjointtriple{\mbox{$G$-spectra}}{i_*}{i^*}{i_!}{\mbox{$H$-spectra}}$$
are modelled by the functors
$$\adjointtriple{\mbox{\cA (G,toral)}}{\theta^!}{\theta_*}{\theta^*}{\mbox{\cA (H,toral)}}.$$
For an object $M$ of $\cA (G,toral)$, $\theta_*M$ is defined on
subgroups $L\subseteq T$ by 
$$(\theta_*M)(L)=\left[ H^*(B T/L)\tensor_{H^*(BW_G^e(L))} M(L)
\right]^{\mW W_H^e(L)}. $$
For an object $N$ of $\cA (H,toral)$, $\theta^*M$ and $\theta^!M$ are
defined on subgroups $K\subseteq S$ by 
$$
(\lambda^*N)(K)=\dichotomy{
\left[    H^*(BT/K)\tensor_{H^*(BW^e_H(K))}N(K)\right]^{\mW W_G^e(K)}
& \mbox{ if $K\subseteq T$}}{0&\mbox{
   otherwise}}$$
and
$$
(\lambda^!N)(K)=\dichotomy{
\left[   \Sigma^{LS/LT} H^*(BT/K)\tensor_{H^*(BW^e_H(K))}N(K)\right]^{\mW W_G^e(K)}
& \mbox{ if $K\subseteq T$}}{0&\mbox{ otherwise}}$$
\end{prop}

\begin{remark}
It is worth making explicit a couple of  special cases. First note
that if  $H=\mN G$ we recover part of Proposition
\ref{prop:modelofresGN}, and similarly, if $H=\mT G$. 

The statement should hold at the derived level even when $G$ and $H$
are of different rank. If so, when $G$ and $H$ are both tori we recover Conjecture \ref{conj:toralchange}.
\end{remark}

\begin{proof}
In view of toral detection and the fact that by Proposition
\ref{prop:Psitheta}
 $\Psi \theta_* =1$, we
can deduce the general case from the torus case. In other words,
writing $V=\mW G$ and $W=\mW H$, and notation given  in the diagram
$$\diagram
\cA (G,toral) \ar[rr]|-{\theta_*} \dto<-1ex>_{\phi_*^G} &&
\cA (H,toral) \dto<-1ex>_{\phi_*^H}
\llto<1.2ex>^{\theta^!}
\llto<-1.2ex>_{\theta^*} \\
\cA (S)[V] \ar[rr]|-{\lambda_*} \uto<-1ex>_{\Psi^G}&&
\cA (T)[W] \uto<-1ex>_{\Psi^H}
\llto<1.2ex>^{\lambda^!}
\llto<-1.2ex>_{\lambda^*} 
\enddiagram$$
we have $\theta^*=\Psi^G \lambda^*\phi^H_*$, 
$\theta_*=\Psi^H \lambda_*\phi^G_*$ and $\theta^!=\Psi^G \lambda^!\phi^H_*$. The formulae are now easily verified.
\end{proof}

\end{document}